\documentclass{article}
\usepackage[
backend=biber,
style=numeric, % Alphabetic default
sorting=ynt
]{biblatex}
\usepackage[document]{ragged2e}
\usepackage{geometry}
\usepackage{graphicx} % Required for inserting images
\newtheorem{theorem}{Theorem}

\newtheorem{lemma}[theorem]{Lemma}
\newtheorem{corollary}{Corollary}[theorem]
\newtheorem{remark}{Remark}
\usepackage{xcolor}
\usepackage{amsmath}
\usepackage{bbm}
\usepackage{longtable}
\usepackage{amssymb}
\usepackage{comment}
\usepackage{hyperref}
\usepackage{relsize}

\newcommand*{\QEDB}{\null\nobreak\hfill\ensuremath{\square}}%
\newcommand*\Eval[3]{\left.#1\right\rvert_{#2}^{#3}}

\newcommand{\sig}{{\sigma}}

\newcommand{\bbE}{{\mathbb E}}

\newcommand{\bbP}{{\mathbb P}}

\usepackage{caption}
\usepackage{blindtext}

\title{Berry-Esseen Bounds for Statistics of Non-Stationary, $\phi$-Mixing Random Variables}
\author{Brendan Williams, Yeor Hafouta}
\date{\today}

\begin{document}
\maketitle
\begin{abstract}
   Using a modification of Stein's method, we generalize the results of Bentkus, G{\"o}tze, and Tikhomirov \cite{bentkus1997berry} to obtain Berry-Esseen bounds for a broad class of statistics of sequences of $\phi$-mixing, non-stationary random variables with polynomial mixing rates. %and linear variance.  
   We then consider applications of this theorem to ensure Berry-Esseen rates for various classes of non-stationary $\phi$-mixing random variables, including rates for a general class of processes of $\phi$-mixing random variables satisfying an aggregate third moment bound.
\end{abstract}

\setlength{\parindent}{13pt} %20 pt is standard

\section{Introduction}
One of the main results in probability theory is the Central Limit Theorem (CLT), which states that an appropriate normalized partial sum with independent summands converges in distribution towards the standard normal law. From a practical point of view a natural question concerns  convergence rates. The classical Berry-Esseen theorem \cite{berry1941accuracy, esseen1942original} provides  optimal rates in the CLT. Since then there have been many works on the CLT together with rates for weakly dependent (i.e. mixing) summands as described in the following.

Berry-Esseen theorems (and CLT) are trivially false in general without some notions of weak dependence.  This notion of weak dependence has been extensively studied in literature. Typically, Berry-Esseen rates are obtained for specific classes of random variables with sufficient mixing rates and structural properties ($\phi$-mixing, $\alpha$-mixing, or $m$-dependence, for example) for the sequence / array, often assuming stationarity. For example, in \cite{jirak2016berry}, Jirak obtained optimal rates for stationary processes for random variables with $p\in(2,3]$ moments. In \cite{rio1996berry}, Rio obtained optimal Berry-Esseen bounds under stationarity and uniformly strong mixing assumptions, with rates $\sum_{j\geq 0}j\phi(j)$ (for the  
$j$-th uniform mixing coefficient $\phi(j)$). See \cite{asghari2017berry}, \cite{wang1999berry}, \cite{wang2006berry}, \cite{bong2022high}, \cite{bolthausen1982berry}, and \cite{lahiri2009berry} for more (of the numerous) examples of Berry-Esseen bounds for various classes of weakly dependent sequences (Berry-Esseen bounds for the kernel density estimator, $m$-dependent U-statistics, weakly negatively dependent random variables, $m$-dependent random samples in high dimensions, a class of Markov chains, and sample quantiles, respectively).

In this paper, we focus on the CLT together with rates  for certain classes of non-linear statistics of the form $F_N(X_1,\cdots,X_N)$ where $\{X_j\}$ is the underlying sequence that represents the sample. Here by non-linear we mean that $F_N$ is not a sum of functions of the individual $X_j$'s. Classical examples include $U$-statistics, studentized sample means and others. Let us note that in all the results discussed above, it is either assumed that the sequence $\{X_j\}$ is independent, or that it is stationary (and weakly mixing). However, in real-life models, both independence and stationarity impose a strong restriction. Indeed, due to obstructions like random noise or external forces it is more common to describe a physical system by a non-stationary sequence. 

In Bentkus et al. \cite{bentkus1997berry} the authors proved CLT rates for 
  for a wide class of statistics formed by weakly dependent exponentially fast mixing stationary sequences $\{X_j\}$. In this paper we extend \cite{bentkus1997berry} to ensure CLT rates for a broader class of statistics with assumptions that are relatively easy to prove in applications. 
To be more precise, we will prove CLT rates for weakly dependent non-stationary sequences $\{X_j\}$. That is, we obtain Berry-Esseen bounds for a class of statistics of sequences of $\phi$-mixing, non-stationary random variables with polynomial mixing rates. In \cite{bentkus1997berry}, Bentkus et al. proved Berry-Esseen bounds for statistics of sums of weakly dependent, stationary random variables; they assumed $\phi$-mixing (uniformly strong mixing) triangular arrays with exponential mixing rates. In this paper, we extend the main theorem of  \cite[Theorem 1.1]{bentkus1997berry}  to obtain the same Berry-Esseen rates for $\phi$-mixing sequences of non-stationary random variables 
for polynomial mixing rates of the order $O(N^{-p})$ for a sufficiently large $p>0$. We also show that by making small modifications to the proof of Theorem \ref{Berry-Esseen Bounds for Statistics of Weakly Dependent Samples Proof} in our paper, we can prove a few strict generalizations of Theorem 1.1 in Bentkus et al. \cite{bentkus1997berry}.

 Recently limit theorems for partial sums of non-stationary weakly dependent sequences have been studied extensively, see \cite{dolgopyat2023berry, hafouta2021convergence, dolgopyat2025berry} and references therein. From that point of view we prove CLT with rates for wide classes of  non-stationary weakly dependent sequences, but instead of considering partial sums we consider more general statistics.

As noted above, dropping the stationarity assumption is also useful for practitioners, as many natural statistical and physical processes are \textit{not} stationary. For example, random walks with and without drift, and time-inhomogeneous Markov chains are non-stationary.  
It is difficult to transform non-stationary processes into stationary processes in general (in particular, stationary and non-stationary sequences do not necessarily have the same convergence rates, or may not converge at all). Also, we only require polynomial mixing rates, which is a strictly weaker assumption than the exponential mixing rates assumed in Bentkus et al. \cite{bentkus1997berry}. %We make the assumption that the variance grows linearly, which is immediately satisfied for stationary sequences.

We argue that the imposed mixing conditions are not so restrictive as to preclude real-world, non-stationary processes. As justification, in Section \ref{Applications of Main Theorem section}, we consider applications of Theorem \ref{Berry-Esseen Bounds for Statistics of Weakly Dependent Samples Proof} in Section \ref{Applications of Main Theorem section} for various classes of random variables. We also demonstrate that in certain applications, our results hold with much weaker restrictions on the growth rates of the variance assumption (under which our main results are formulated), see Section  \ref{Applications of Main Theorem section}.

\section{Preliminaries and main results under linear growth of variance}

Let $\{X_i\}_{i\in \mathbb{Z}^{+}}$ be a sequence (or array) of random variables taking values in an arbitrary measurable space $(\mathcal{X}, \mathcal{A})$, and    denote by $\sigma[a,b]$ the $\sigma$-algebra generated by the random variables $X_{j}$ such that $j\in [a, b]\cap \mathbb{Z}^+$.    As a measure of dependence between random variables in this sequence, for each $m\in \mathbb{Z}^+$ define \begin{align}
    \phi(m) := \sup_{k\in \mathbb{Z}^{+}}\sup_{A\in\sigma[1,k]}\sup_{B\in \sigma[k+m, +\infty)}\left|\bbP(B\,|\, A) - \bbP(B)\right|.\label{phi mixing definition}
\end{align}
$\phi(m)$ is referred to as the $m$-th $\phi$-mixing coefficient.%\footnote{There are other ways in which the everyday concept of mixing  is quantified in the literature. The relevant forms of mixing measure the dependence of an indexed family of random variables, and have measure-theoretic definitions. 
%Furthermore, for Decoupling Inequality \ref{Decoupling}, we could have used stronger mixing conditions, such as $\psi$-mixing or $m$-dependence (see \cite{bradley2005basic} for details about the different types of mixing). $\phi$-mixing coefficients were thus used in \cite{bradley2005basic} for greater generality, while still ensuring that Lemma \ref{Decoupling} is valid. For example, $\alpha$-mixing (strong mixing) sequences do \textit{not} satisfy Lemma \ref{Decoupling}.}.  

 We assume that for all $m\geq 1$, \begin{align*}
\phi(m)\leq Km^{-p}
\end{align*} for a constant $K>0$ that is independent of $N$, and a sufficiently large $p>0$. Roughly speaking, these polynomial rates ensure that the dependence between the random variables in the sequence decays fast enough that functions of random variables far enough apart in the sequence only differ in expectation from their independent copies by a controlled amount.%\footnote{Notably, $\phi(m)\to 0$ as $m\to +\infty$, so we say that $\{X_i\}_{i\in \mathbb{Z}^{+}}$ is ``$\phi$-mixing" or ``uniformly strong mixing." Thus, we can think of mixing processes as being asymptotically independent, in a certain sense (in this case in terms of $\phi(m)$).
%}.  

\subsection{Statement of Berry-Esseen Theorem for Statistics of Non-Stationary, Weakly Dependent Samples}
We set up the main theorem now before stating it. Let $\{X_j\}_{j\in \mathbb{Z}^{+}}$ be a sequence of random variables taking values in an arbitrary measurable spaces $(\mathcal{X}_j, \mathcal{A}_j)$. We consider statistics $T = t(X_1, \cdots,X_N)$  of the form,
\begin{align*}
    T = S+R \text{, where } S =S_N= \sum_{j=1}^N g_j(X_j)
\end{align*}
for some $g_j: \mathcal{X}_j \to \mathbb{R}$ such that $\mathbb{E}[g_j(X_j)]=0$ for each $j$, $R = R_N(X_1,\cdots,X_N)$ is the remainder term, and $S$ is the linear component of the statistic.  In this representation, $t$, $T$, $S$, $g_j$, and $R$ may depend on $N$, but this dependence will be left implicit. Here, $t=t_N$ is a real-valued function of $N$ variables.

We now impose some restrictions on $R$ and $g_j$. Define $\rho_k(N) := \sup_{j\in\{1, \cdots, N\}}\mathbb{E}[|g_j(X_j)|^k]$, and define $\sigma^2_N:= \mathbb{E}[S^2]$. We make the following assumptions:\\

\noindent\fbox{\begin{minipage}{\textwidth}
\paragraph{Assumptions} 
\begin{enumerate}
    \item $\phi(m)\leq Km^{-p}$ for all $m\in \mathbb{Z}^{+}$, where $K>0$ is a constant independent of $N$ and $p>0$ is sufficiently large.
\item The variance of $S_N=\sum_{j=1}^N g_j(X_j)$ grows linearly over $N$; that is, there exist positive constants $b$ and $\Tilde{b}$ and an $N_0\in \mathbb{Z}^{+}$ such that $\tilde{b}N \leq \text{Var}(S_N) \leq bN$ for all $N\geq N_0$. \label{assumption 2}
\item There exists $\rho>0$ such that $\sup_{N\in \mathbb{N}}\rho_3(N) \leq \rho<+\infty$.
\item There exists $\Sigma>0$ such that $\sigma_N^2= \text{Var}(S_N)\geq \Sigma^2>0$.
\end{enumerate}
\end{minipage}}

Note that the supremum over $\rho_3$ is required due to the implicit dependence of $g_j$ on $N$. These conditions ensure that the mixing occurs fast enough to ensure Berry-Esseen convergence rates.

We state the main theorem of this paper now.

\begin{theorem}\label{Berry-Esseen Bounds for Statistics of Weakly Dependent Samples Proof} {\normalfont \textbf{(Berry-Esseen Bounds for Statistics of $\phi$-Mixing, Non-Stationary Sequences, with Polynomial Rates and Linear Variance)}} 

Assume that the above four assumptions hold. For each $1\le j\leq k\leq N$, let $$R_{j,k} = R_{j,k}(X_1,\cdots, X_{j-1}, X_{k+1}, \cdots,X_N),\hspace{.1in} (j, k \in \{1,\cdots,N\}),$$ denote a $\sigma^C[j,k]$-measurable random variable such that $R_{j,j-1} = R$ for all $2 \leq j\leq N$. For each 
$N\in\mathbb{Z}^{+}$, define, for any $\varepsilon >0$, \begin{align*}
    \gamma=\gamma(N) := \max\left\{ \mathbb{E}\left[\left|R_{j,k} - R_{j,k-1}\right|^{3/2}\right]^{2/3}: |j-k|\leq N^{\varepsilon}  \text{ and } 1\leq j\leq k \leq N \right\}.
\end{align*} Then, for a choice of $p$ large enough, we have,
\begin{align*}
    \sup_{x\in \mathbb{R}}\left|\bbP(T<\sigma_{N}x) -\Phi(x)\right|\leq AN^{-1/2+\varepsilon_p} + AN^{-1/2}\,\mathbb{E}\left[R\right] + A\gamma N^{\varepsilon_p},
\end{align*}  where $A = A(p, K, \Sigma, \rho)$ is a constant independent of $N$ and $\varepsilon_p\to 0$ as $p\to\infty$ (both $A$ and $\varepsilon_p$ can be recovered from the proof).
\end{theorem}
\begin{remark}
This theorem is a generalization of \cite[Theorem 1.1]{bentkus1997berry}.
If we assume that $\phi(m)$ decays exponentially fast as $m\to\infty$ then we can recover the logarithmic rates in 
 \cite[Theorem 1.1]{bentkus1997berry}. That is, we can show that \begin{align}
    \sup_{x\in \mathbb{R}}\left|\bbP(T<\sigma_{N}x) -\Phi(x)\right|\leq AN^{-1/2}(\log \,(N))^2 + AN^{-1/2}\,\mathbb{E}\left[R\right] + A\gamma\,(\log\, (N))^3,\label{Main Theorem for exponential rates}
\end{align}
where in the definition of $\gamma$ we can replace $N^\varepsilon$ with $(\log\,(N))^3$.
\end{remark}
 As noted in the previous section, in Sections \ref{Stationary extension section} and \ref{Bounded} under appropriate conditions we are able to remove the growth rate $N$ of the variance. However, this is done by an application of the above theorem, and so we decided to formulate the main result under the linear growth assumption. %Let us note that without some growth assumptions on the variance the only results in literature that seem to provide some CLT rates for weakly dependent variables are \cite{hafouta2023convergence}

The proof of  \cite[Theorem 1.1]{bentkus1997berry} only uses the fact that the random variables are stationary in the latter half of their proof. As such, the basic structure and several of the arguments in the proof of Theorem \ref{Berry-Esseen Bounds for Statistics of Weakly Dependent Samples Proof} in this paper are left either fundamentally unaltered or are modified from the proof of  \cite[Theorem 1.1]{bentkus1997berry}. Furthermore, only minor changes have been made to the notations of \cite{bentkus1997berry}, with most of the modifications arising from the challenge of considering non-stationary sequences. We state the setup for the proof of Theorem \ref{Berry-Esseen Bounds for Statistics of Weakly Dependent Samples Proof} now.

The proof of the above Berry-Esseen theorem  section uses a modification of Stein's method, in which we prove that the ordinary differential equation,\begin{align}\label{Perturbation}
f'(t) = -tf(t) + \varepsilon(t)f(t) + \varepsilon_0(t), \text{    } f(0)=1,
    \end{align} is satisfied for a pair of ``sufficiently small" functions $\varepsilon(t)$ and $\varepsilon_0(t)$ over an appropriate interval around 0.
    This differential equation can be viewed as a perturbation of the differential equation solved by the standard normal characteristic function, \begin{align}\label{Normal Char Func Diffeq}
        f'(t) = -tf(t) \text{,       } f(0)=1.
    \end{align} It is known that if $f$  solves Equation \ref{Perturbation} on a ``large enough" interval of $t$, then  $f$ also uniquely solves the differential equation for the normal distribution (Equation \ref{Normal Char Func Diffeq}). The setup for this application of Stein's method is identical to that of \cite{bentkus1997berry}, and thus the detailed proof will not be included. The version of Stein's method used by Bentkus et al. in \cite{bentkus1997berry}, and that we use in the proof of Theorem \ref{Berry-Esseen Bounds for Statistics of Weakly Dependent Samples Proof}, was developed by modifying the techniques used by Tikhomirov \cite{tikhomirov1980rate} and Stein \cite{stein1972bound}, among others.

In the proof of Theorem \ref{Berry-Esseen Bounds for Statistics of Weakly Dependent Samples Proof} in our paper, we do not use the definition (\ref{phi mixing definition}) of $\phi$-mixing directly, and instead use the following decoupling inequality, valid for $\phi$-mixing sequences. Recall that the expectations for linear statistics of random variables can be decoupled if they are independent, and expectations of functions of independent random variables can more easily be calculated.
 This lemma allows us to quantify the fact that the expectation of functions of the random variables that are ``almost" independent are almost the same as that of their independent copies. The following lemma says that a $\phi$-mixing sequence obeys this property, up to a ``small enough" error term. 

We state the decoupling lemma that we will use now.
\begin{lemma}[Decoupling Inequality]\label{Decoupling}
        Let $\xi$ be a $\sigma[1, k]$-measurable random variable and let $\eta$ be a $\sigma[k+m, +\infty)$-measurable random variable, in the same $\phi$-mixing family $\{X_j\}_{j\in \mathbb{Z}^{+}}$. Also assume that $\eta$ and $\xi$ take values in a Polish space\footnote{This condition is more general than we require. In our proof, each of our choices of $\eta$ and $\xi$ are complex-valued, and $\mathbb{C}$ is a Polish space.}. Let $\hat{\eta}$ denote an identically distributed copy of $\eta$ that is also independent of $\xi$, and let $\nu$ be an arbitrary measurable function such that $\sup_{u,v}| \nu(u,v) |\leq D$. Then, \begin{align*}
            \left|\mathbb{E}\left[\nu(\xi, \eta)\right] - \mathbb{E}\left[\nu(\xi, \hat{\eta})\right]\right|\leq D\phi(m).
        \end{align*}
    \end{lemma}
This lemma is general enough to provide bounds for non-linear statistics, which will be necessary in the proof of Theorem \ref{Berry-Esseen Bounds for Statistics of Weakly Dependent Samples Proof}. 

\section{Proof of Theorem \ref{Berry-Esseen Bounds for Statistics of Weakly Dependent Samples Proof}}
 Denote by $\sigma[a,b]$ the $\sigma$-algebra generated by the random variables $X_{j}$ such that $j\in [a, b]$, and denote by $\sigma^C[a,b]$ the $\sigma$-algebra generated by the random variables $X_{j}$ such that $j\in \{1,\cdots,N\} \setminus [a, b]$. If $b<a$, then set $\sigma^C[a, b] = \sigma(X_1,\cdots, X_N)$. In the proof of Theorem \ref{Berry-Esseen Bounds for Statistics of Weakly Dependent Samples Proof}, we let $a$, $A$, $A_1$, $C$, $C_1$, etc., denote positive constants that are independent of $N$ (and may be chosen implicitly in the proof) and depend on the parameters of interest ($K$, $\rho$, $\Sigma$, $L$, $\gamma$, etc.). The numerical values of the constants that are chosen throughout the proof will not be tracked (and will be dependent upon the application). In the proof of Theorem \ref{Berry-Esseen Bounds for Statistics of Weakly Dependent Samples Proof}, we treat $p$ as being  a fixed, sufficiently large constant.

\textbf{Proof of Theorem \ref{Berry-Esseen Bounds for Statistics of Weakly Dependent Samples Proof}}: While proving this theorem, we shall assume that $|g_j(X_j)| \leq \sqrt{N}$ for each $N\in \mathbb{Z}^{+}$ and each $j$. Otherwise, we replace $g_j(X_j)$ by, 
$$
g_j(X_j)\mathbbm{1}_{\left\{\left|g_j(X_j)\right|\leq \sqrt{N}\right\}}-\mathbb{E}\left[g_j(X_j)\mathbbm{1}_{\left\{\left|g_j(X_j)\right|\leq \sqrt{N}\right\}}\right].
$$ 
We refer to Section \ref{TruncSec} for the exact details of this truncation procedure.

Let $a$ denote a sufficiently small positive constant that may depend only on $p$, $K$, $\Sigma$, and $\rho$ (a sufficiently small choice of $a$ will be chosen throughout the proof). Also, define $f(t) := \mathbb{E}[e^{itT/\sqrt{N}}]$ and $\varphi(t) := e^{-t^2/2}$. 

To simplify the arguments,  assume that $\sigma_N^2 = \mathbb{E}[S^2] = N$. Define the ``step size" $m := A_{p} N^{\epsilon_{p}}$, where $A_{p}$ is a sufficiently large constant depending only on $p$, and $\epsilon_{p}$ is a small positive constant depending only on $p$. We will show that $f$ over the interval $|t|\leq a\sqrt{N}/m^2$ (assuming without loss of generality that $N\geq 2$) $f(t) $ satisfies the ordinary differential equation, \begin{align}
f'(t) = -tf(t) + \varepsilon(t)f(t) + \varepsilon_0(t), \text{    } f(0)=1, \label{Perturbed Normal diffeq}
    \end{align}
where \begin{align}
    |\varepsilon(t)| \leq \frac{Am^2t^2}{\sqrt{N}}, \label{epsilon}
\end{align} and \begin{align}
    \left|\varepsilon_0(t)\right|\leq A |t|m^2\gamma + A\mathbb{E}\left[|R|\right] + Am|t|N^{-1/2} + AmN^{-1/2} + AN^{-1/2}\label{epsilon_0}.
\end{align}
Differential Equation \eqref{Perturbed Normal diffeq} has the unique solution
\begin{align*}
    f(t) = \varphi(t)\exp \left\{\int_{0}^t \varepsilon(u) \,du\right\} + \varphi(t)\int_{0}^{t}\exp \left\{\frac{u^2}{2} +\int_{u}^{t} \varepsilon(z)\,dz  \right\}\varepsilon_0 (u)\, du.
\end{align*} Then $|f(t) - \varphi(t)|\leq I_1 + I_2$ by the triangle inequality,
where \begin{align}
    I_1 &:= \varphi(t)\left|e^{\int_{0}^t \varepsilon(u) \,du} - 1\right|\\
    &\leq Am^2N^{-1/2}t^2e^{-t^2/4},\label{I_1}
\end{align} (by the mean value theorem) and \begin{align}
    I_2 &:= \varphi(t)\int_{0}^{t}\exp \left\{\frac{u^2}{2} +\int_{u}^{t} \varepsilon(z)dz  \right\}|\varepsilon_0 (u)|\,du\\
    &\leq \left[A\left(m^2 + mN^{-1}\right)\,\gamma \min\, \{1, |t|\}\right] + AN^{-1/2}\left(\mathbb{E}\left[|R|\right] +1 \right)\min\left\{|t|^{-1}, |t|]\right\} .\label{I_2}
\end{align} Combining Inequalities \ref{I_1} and \ref{I_2} for $I_1$ and $I_2$ with the Berry-Esseen theorem for characteristic functions (i.e., the smoothing inequality), we get the result of the theorem\footnote{To bound  Inequalities \ref{I_1} and \ref{I_2} we must split the domain of the integrals into the regions $[0,(\sqrt{N}/m^2)^{1/3}]$ and $((\sqrt{N}/m^2)^{1/3}, a\sqrt{N}/m^2]$. The integral over the first interval is controlled by a crude bound. The argument is unchanged from Bentkus et al. \cite{bentkus1997berry}, as it does not rely on stationarity.}. Hence, we just need to show that $f$ satisfies the ordinary differential equation \ref{Perturbed Normal diffeq}. To verify that Inequality \ref{I_1} is valid, we may apply the  inequality $|e^t - 1| \leq |t|e^{|t|}$ on the real interval $|t|\leq a\sqrt{N}/ m^2$ for a sufficiently small $a$. Indeed, $$\left|\int_{0}^t\varepsilon(u)\,du\right|\leq \left|\int_{0}^t 
Am^2t^2N^{-1/2}\, du\right| = Am^2\frac{t^3}{3}N^{-1/2} = t\left(Am^2\frac{t^2}{3}N^{-1/2}\right).$$ Also, for a sufficiently small choice of $a$,
\begin{align*}
\exp\left\{\left|\int_{0}^t\varepsilon(u)\,du\right|\right\}\leq e^{t^2/12},
\end{align*} which implies this result.

To verify that Inequality \ref{I_2} is valid, we use the fact that,\begin{align*}\left|\int_{u}^{t}\varepsilon(z)\,dz\right| \leq \Eval{\frac{Am^2z^3}{3\sqrt{N}}}{u}{t}  \leq aA(t^2-u^2)<\frac{1}{4}(t^2-u^2),
\end{align*} which we derive from the facts that $|t|\leq a\sqrt{N}m^{-2}$, $|\varepsilon(t)| \leq \frac{Am^2t^2}{\sqrt{N}}$, and that $a$ is chosen small enough so that $aA<1/4$.

Furthermore, to obtain the desired result we now only need to show that $T$ satisfies Inequalities \ref{Perturbed Normal diffeq}, \ref{epsilon}, and \ref{epsilon_0}  over the interval $|t|\leq a\sqrt{N}/m^2$.

We write $B\approx D$ if $B = D + \varepsilon(t)f(t) + \varepsilon_0(t)$ on the interval $|t|\leq a\sqrt{N}m^{-2}$, where the functions $\varepsilon$ and $\varepsilon_0$ satisfy Inequalities \ref{epsilon} and \ref{epsilon_0}. In particular, we must  prove that $f'(t)\approx -tf(t)$. 

Differentiating $f$, we get,
\begin{align*}
    f'(t) &= \frac{i}{\sqrt{N}} \,\mathbb{E}\left[Se^{itT/\sqrt{N}}\right] + \frac{i}{\sqrt{N}}\,\mathbb{E}\left[Re^{itT/\sqrt{N}}\right]\\
    &\approx I_3 := \frac{i}{\sqrt{N}}\, \mathbb{E}\left[Se^{itT/\sqrt{N}}\right]
    \\&= \frac{i}{\sqrt{N}}\sum_{j=1}^{N}\mathbb{E}\left[g_j(X_j)e^{itT/\sqrt{N}}\right].
\end{align*} The error in this approximation is bounded by $\mathbb{E}[|R|]/\sqrt{N}$. This approximation corresponds to $\varepsilon_0(t)$.

Now, define $S_{j, 0} := N^{-1/2}S$, and define,\begin{align*}
    \Delta_{j, 1} := N^{-1/2}\sum_{\ell \in \Omega_1^{(j)}}g_\ell(X_{\ell}), \text{   where   } \Omega_1^{(j)} := \{\ell: 1\leq \ell \leq N, \,|\ell - j|\leq m\}.
\end{align*}Define $S_{j,1} := N^{-1/2}S - \Delta_{j, 1}$.  In general, we define,
\begin{align*}
    \Delta_{j, s} := N^{-1/2}\sum_{\ell \in \Omega_s^{(j)}}g_\ell(X_{\ell}), \text{   where   } \Omega_s^{(j)} := \{\ell: 1\leq \ell \leq N,\, (s-1)m < |\ell - j|\leq sm\},
\end{align*}  and define $S_{j, s} := S_{j,s-1} - \Delta_{j,s}$ by induction.  Also, define $r:= C_p \log\, (N)$  for a constant $C_p$ independent of $N$ ($C_p$ will be chosen sufficiently large throughout the proof).
 For each $j$, define, \begin{align*}
    Q_j := R_{j-rm, j+rm},\hspace{.5in} \delta_j:= N^{-1/2}R-Q_j, \hspace{.5in}  T_{j, s}:= S_{j, s} + Q_j.
\end{align*} We will apply the inequality $(a_1 + \cdots + a_p)^{3/2} \leq \sqrt{p}(a_1^{3/2}+\cdots + a_k^{3/2})$ (valid for  $a_j\in \mathbb{R}^+$ and $k\in \mathbb{Z}^+$), with the triangle inequality. We use the assumption that $R = R_{j-rm, j-rm-1}$, and that there are $2rm+1$ summands in the following expression, and that $r<m$ for large enough $N$. This yields,
\begin{align}
    \mathbb{E}\left[|\delta_j|^{3/2}\right] &= \mathbb{E}\Big[\big|(N^{-1/2}R - R_{j-rm, j-rm}) + (R_{j-rm, j-rm} - R_{j-rm, j-rm+1})\\
    &\hspace{.3in}+ \cdots + (R_{j-rm, j+rm-1}-R_{j-rm, j+rm})\big|^{3/2}\Big]\\
    &\leq \sum_{i = -rm}^{rm} \mathbb{E}\left[\left|R_{j-rm, j+i} - R_{j-rm, j+i-1}\right|^{3/2}\right]\\
    &\leq (2rm+1)(2rm+1)^{1/2}\,\gamma^{3/2}\\
    &\leq Am^3(N^{-1/2}\gamma)^{3/2},\label{delta j}
\end{align} valid for a sufficiently large constant $A = A(A_p, C_p)$ and by our choices of $r$ and $m$. Then, by the mean value theorem applied to $e^{\delta_j}$ (centered at 0),
\begin{align*}
    I_3: &=\frac{i}{\sqrt{N}}\sum_{j=1}^{N}\mathbb{E}\left[g_j(X_j)\,e^{itT/\sqrt{N}}\right]\\
    &=\frac{i}{\sqrt{N}}\sum_{j=1}^N \mathbb{E}\left[g_j(X_j)\,e^{it(T_{j, 0} + \delta_j)}\right]\\
    &\approx I_4 := \frac{i}{\sqrt{N}}\sum_{j=1}^{N}\mathbb{E}\left[g_j(X_j)\,e^{itT_{j, 0}}\right].
\end{align*} The error in this approximation is bounded by,
\begin{align*}
    \left|\frac{i}{\sqrt{N}}\sum_{j=1}^{N}\mathbb{E}\left[g_j(X_j)\,e^{it(T_{j, 0} + \delta_j)} - g_j(X_j)\,e^{itT_{j, 0}}\right]\right|&\leq \left|\frac{i}{\sqrt{N}}\sum_{j=1}^{N} \mathbb{E}\left[g_j(X_j)\,e^{itT_{j,0}} \, C\,|\delta_j|\right]\right|\\
   &\leq \frac{C}{\sqrt{N}} \sum_{j=1}^{N} \mathbb{E}\left(\left[|\delta_j|^{3/2}\right]^{2/3}\,\mathbb{E}\left[|g_j(X_j)|^3\right]^{1/3}\right)\\
   &\leq \frac{1}{\sqrt{N}} \sum_{j=1}^{N}\left(A|t|m^2(\gamma N^{-1/2})\right) \rho^{1/3}\\
   &\leq A|t|m^2\gamma.
\end{align*} We obtain the first inequality from the Taylor remainder theorem for powers of $\delta_j$ (i.e. the mean value theorem). The second inequality is obtained by the triangle inequality and an application of the H\"{o}lder inequality with coefficients 3 and 3/2. The third inequality is obtained from an application of Inequality \ref{delta j}, the assumption that $\sup_{N\in \mathbb{N}}\rho_3 \leq \rho<+\infty$, and the fact that $|e^{ix}| = 1$ for all $x\in\mathbb{R}$. The final inequality uses the fact that there are $N$ summands. This approximation corresponds to $\varepsilon_0 (t)$. 

Splitting $T_{j, 0} = T_{j, 1} + \Delta_{j, 1}$, we then have the decomposition,
\begin{align*}
    I_4 = \frac{i}{\sqrt{N}}\sum_{j=1}^{N}\mathbb{E}\left[g_j(X_j)J_0^{(j)} e^{itT_{j, 1}}\right]+\frac{i}{\sqrt{N}}\sum_{j=1}^{N}\mathbb{E}\left[g_j(X_j)J_1^{(j)} e^{itT_{j, 1}}\right], 
\end{align*}where $J_0^{(j)} := 1$ and $J_s^{(j)} := e^{it\Delta_{j, s}}-1$ for each $s\in \{1,\cdots ,N\}$. Repeating this process inductively, we obtain that for each $r\geq 2$ ($r= C_p \log (N) $ will be chosen later for a sufficiently large choice of $C_p$),\begin{align}\label{I_4 Inequality}
    f'(t)\approx I_4 = \sum_{s=1}^{r-1} \sum_{j=1}^{N} I(j, s) + \sum_{j=1}^{N} I_1(j, r), 
\end{align}
where 
\begin{align*}
    I(j, s) &:= \frac{i}{\sqrt{N}}\mathbb{E}\left[g_j(X_j)J_1^{(j)} \cdots J_s^{(j)} e^{itT_{j, s+1}}\right],\end{align*}
and
    \begin{align*}
    I_1(j, r) &:= \frac{i}{\sqrt{N}}\mathbb{E}\left[g_j(X_j)J_1^{(j)} \cdots J_r^{(j)} e^{itT_{j, r}}\right].
\end{align*}

The expectation of $I(j,s)$ is taken over a set of random variables with vertices $\ell$ such that $sm < |\ell - j|\leq (s+1)m$. We can thus apply the decoupling inequality (Inequality \ref{Decoupling}) with the functions 
$$
\xi = g_j(X_j)J_1^{(j)} \cdots J_r^{(j)}
$$
and $\eta = e^{itT_{j, r}}$, and $\nu$ defined by $\nu(\xi, \eta) = \xi\eta$. By our construction, we have that $g_j(X_j)J_1^{(j)} \cdots J_r^{(j)}\in \sigma(X_{\ell}: \ell\in \bigcup_{k=1}^r \Omega_k^{(j)})$ and $e^{itT_{j, r}}\in \sigma(X_{\ell}: \ell\notin \bigcup_{k=1}^{r+1} \Omega_k^{(j)})$. In particular, the distance between the indices $\ell$ in these two $\sigma$-algebras is at least $2m>m$. We also use the facts that $|J_{\ell}^{(j)}|\leq 2$ for all $j$ and $\ell$, and $g_j(X_j)\leq \sqrt{N}$ for each $j$. Thus, because there are $N$ summands for each $s$ in the inner sum, we have, 
\begin{align}
    \sum_{s=1}^{r-1} \sum_{j=1}^{N} I(j, s) \approx \sum_{s=1}^{r-1} \sum_{j=1}^{N} I_2(j, s), \hspace{.2in}\text{ where }  I_2(j,s):= \frac{i}{\sqrt{N}}\mathbb{E}\left[g_j(X_j)J_1^{(j)}\cdots J_s^{(j)}\right]\mathbb{E}\left[e^{itT_{j,s+1}}\right].
\end{align} By Lemma \ref{Product Lemma} for products in strong mixing (or $\phi$-mixing) sequences, the error in this approximation is bounded above by,
\begin{align}
    A \sum_{s=1}^{r-1}N2^s\phi(m)\leq AN2^r\phi(m)\leq AN^{-1/2},\label{4.9 Inequality}
\end{align} for a sufficiently large constant $A = A(A_p, C_p)$, $m = A_p N^{\epsilon_p}$, and $r =  C_p \log \,(N)$, and for a  sufficiently large $p$  (we also make the trivial assumption that $N\geq 2$). Here, we ensure that $A$ is large enough so that $A \log\,(2)>1$. This estimate corresponds to $\varepsilon_0 (t)$. Indeed, since  $A_p$ is a constant independent of $N$, we may use the fact that $2^{C_p \log (N)}\leq e^{C_p \log (N)}=N^{C_p}$. Thus, we may choose a sufficiently large constant $A$ so that $N^{3/2 + A}\leq C(A\log (N))^p$, which is possible because $A\log\, (N)\geq A \log\,(2)>1$. 

 We have thus shown that, $$f'(t) \approx \frac{i}{\sqrt{N}}\sum_{s=1}^{r-1}\sum_{j=1}^{N}\mathbb{E}\left[g(X_j)J_1^{(j)} \cdots J_{s}^{(j)}\right]\,\mathbb{E} \left[e^{itT_{j, s+1}}\right].$$

We then apply the mean value theorem, and again using the assumption that $\sup_{N\in \mathbb{N}}\rho_3 \leq \rho<+\infty$, and the triangle inequality, 
we can write $|J_s^{(j)}|\leq |t\Delta_{j,s}|$. This yields,  \begin{align}\label{Jsj inequality}
    \left\|J_s^{(j)}\right\|_{L^3}\leq |t| \left\|\Delta_{j,s}\right\|_{L^3}\leq C\frac{|t|m}{\sqrt{N}}.
\end{align}

Next, we group the even and odd $J_{i}^{(j)}$ into separate groups and apply the H\"{o}lder inequality with coefficients 2 and 2. This gives, \begin{align}
    \left(\mathbb{E}\left[|J_1^{(j)} \cdots J_s^{(j)}|^{3/2}\right]\right)^{2/3} \leq \mathbb{E}\left[\left|\prod_{i \text{ even}}^s J_{i}^{(j)}  \right|^3\right]^{1/3}     \mathbb{E}\left[\left|\prod_{i \text{ odd}}^s J_{i}^{(j)}  \right|^3\right]^{1/3}.\label{Js Inequality 2}
\end{align}

We now use the inequality $|J_{i}^{(j)}|\leq 2$, and replace the $J_{i}^{(j)}$ by their independent copies using Lemma \ref{Decoupling} (the decoupling lemma) a total of $s-2$ times to decouple both expectations in Inequality \ref{Js Inequality 2}. Observe that in each expectation, the $X_j$ terms that $J_{i}^{(j)}$ is dependent upon are a distance of at least $2m>m$ indices apart ($J_{i}^{(j)}\in$ $\sigma(\Omega_{i}^{(j)})$ for each $i$ and $j$). This gives,
\begin{align}\label{4.11 Inequality}
    \left(\mathbb{E}\left[|J_1^{(j)} \cdots J_s^{(j)}|^{3/2}\right]\right)^{2/3} \leq \left\{\prod_{\ell = 1}^{s}\left(\mathbb{E}[|J_{\ell}^{(j)}|^3]\right)^{1/3}\right\} + As2^s\phi^{1/3}(m).
\end{align} Thus, by our choices of $m = A_{p} N^{\epsilon_{p}}$ and $r= C_p\log (N)$, along with the fact that $2<e$, and sufficiently large choices of $A$ and $p$, we have, \begin{align*}
    2^s\phi^{1/3}(m)
    &\leq e^{2r}\phi^{1/3}(m) \\
    &\leq N^{2C_p}\phi^{1/3}(m) \\
    &\leq AN^{-3}.
\end{align*}  Here, we used the same calculation for $\phi(m)$ as before, valid for sufficiently large $K$ and $p$. This may require a larger choices of $K$ and $p$, and we may choose $A_p = A_p(C_p)$ sufficiently large. 

Next, we apply Inequality \ref{Jsj inequality} to Inequality \ref{4.11 Inequality}, and then use the bound $|t|\leq a\sqrt{N}/m^2$. This yields,
\begin{align}
    \left(\mathbb{E}\left[|J_1^{(j)} \cdots J_s^{(j)}|^{3/2}\right]\right)^{2/3} &\leq A\left(|t|m N^{-1/2}\right)^s + AN^{-3}\\
    &\leq A\left(a/m\right)^s+ AN^{-3}.\label{Ji inequality}
\end{align}

Indeed, for Inequality \ref{4.11 Inequality}, $s\leq r = A_p \log (N)$, so \begin{align*}
    A3^s(\phi(m))^{1/3}&\leq A3^r(\phi(m))^{1/3} \leq A3^{C_p \log (N)}(\phi(m))^{1/3} \\
    &= A e^{C_p\log (N) \log (3)}  (\phi(m))^{1/3} = A N^{C_p\log(3) }  (\phi(m))^{1/3}.
\end{align*}  Combining these inequalities, we have, \begin{align}
    N^{-1/2}\left|\sum_{s=2}^{r-1}\sum_{j=1}^N\,\mathbb E\left[g_j(X_j)J_1^{(j)}\cdots J_s^{(j)}\right]\mathbb E\left[e^{itT'_{j,s+1}}\right] \right| &\leq N^{-1/2} \left|\sum_{s=2}^{r-1}\sum_{j=1}^N \|g_j(X_j)\|_{L^3} \,  \|J_1^{(j)}\cdots J_s^{(j)}\|_{L^{3/2}}\right|\\
    &\leq N^{-1/2} \sum_{s=2}^{r-1}\sum_{j=1}^N (\rho^{1/3})\left(A\,(a/m)^s+ AN^{-3}\right)\\
    &\leq N^{-1/2} \sum_{s=2}^{r-1}\sum_{j=1}^N \left(A\left(a/m\right)^s+ AN^{-3}\right).\label{Jl decomposition}
\end{align}
Furthermore, by Inequality \ref{I_4 Inequality},
\begin{align}
    f'(t)\approx \sum_{s=1}^{r-1}\sum_{j=1}^N I_2(j,s) \hspace{.2in}\text{because }\hspace{.2in} \sum_{j=1}^N I_1(j,r)\approx 0.
\end{align}

To prove the above approximation, we use the fact that $I_2(j,0)=0$ for each $j$ (because $\mathbb{E}[X_j] = 0$ for each $j$), and use Inequality \ref{4.11 Inequality} to bound $I_1(j,r)$. The error in this approximation is bounded by $AN^{-1/2}$.
Define $T'_{j,s} := T_{j,s}+\delta_j = S_{j,s} + R$, and define $\mu_{j,s} := S/\sqrt{N} - S_{j,s+1}$. Then $T'_{j,s+1} = T/\sqrt{N}-\mu_{j,s}$. We may thus write,
\begin{align}\label{I_3(j,s)}
    f'(t) \approx \sum_{s=1}^{r-1}\sum_{j=1}^N \frac{i}{\sqrt{N}} \mathbb{E}\left[g_j(X_j)J_1^{(j)}\cdots J_s^{(j)}\right]\mathbb{E}\left[e^{itT'_{j,s+1}}\right].
\end{align} The error in this approximation is  bounded by $A|t|m^2 N^{-1}\gamma$. We obtain this bound by expanding in powers of $\delta_j$ (i.e. the mean value theorem), and applying the inequality $\mathbb{E}[|\delta_j|^{3/2}]\leq Am^3 \gamma^{3/2}$ along with the H\"{o}lder inequality with coefficients 3 and 3/2, and Inequality \ref{Ji inequality}. 

Next, we show that $N^{-1/2} \sum_{s=2}^{r-1}\sum_{j=1}^N\mathbb E[g_j(X_j)J_1^{(j)}\cdots J_s^{(j)}]\,\mathbb E[e^{itT'_{j,s+1}}]\approx 0$. To this end, we now decompose the sum as,
\begin{align}
    N^{-1/2}\left|\sum_{s=2}^{r-1}\sum_{j=1}^N\mathbb E\left[g_j(X_j)J_1^{(j)}\cdots J_s^{(j)}\right] \left(\mathbb{E}\left[e^{itT'_{j,s+1}}\right] - \mathbb{E}\left[e^{itT/\sqrt{N}}\right]\right)\right| \\= N^{-1/2}\left|\sum_{s=2}^{r-1}\sum_{j=1}^N\mathbb{E}\left[g_j(X_j)J_1^{(j)}\cdots J_s^{(j)}\right]\left(D_1^{(s,j)} f(t) + D_2^{(s,j)}\right)\right|, \label{D1 D2 decomposition}
\end{align} where $D_1^{(j,s)} := \mathbb{E}[e^{-it\mu_{j,s}}-  1]$ and $D_2^{(j,s)} := \mathbb{E}\left[(e^{-it\mu_{j,s}}-\mathbb{E}[e^{-it\mu_{j,s}}])\,e^{itT/\sqrt{N}}\right]$. In this decomposition, we used the fact that $\mathbb{E}[e^{itT_{j,s+1}}] - \mathbb{E}[e^{itT/\sqrt{N}}] = D_1^{(s,j)}\,\mathbb{E}[e^{itT/\sqrt{N}}] + D_2^{(s,j)}$. By the triangle inequality, we may consider the $D_1^{(s,j)}$ and $D_2^{(s,j)}$ terms in Equation \ref{D1 D2 decomposition} separately, and show that they are both negligible. 

We first demonstrate that the sum of the  $D_1^{(j,s)}$ terms in Equation \ref{D1 D2 decomposition} is negligible. The $D_1^{(j,s)}$ terms correspond to $\varepsilon(t)$. By the mean value theorem, we deduce that,
\begin{align}
    \mathbb{E}\left[e^{-it\mu_{j,s}}\right]-1 \leq A|t| \,\mathbb{E}\left[|\mu_{j,s}|\right]\leq A|t|msN^{-1/2}.\label{mean value theorem bound}
\end{align} Recall the bounds from Inequality \ref{Ji inequality} that $\mathbb{E}[|g_j(X_j) J_{1}^{(j)} \cdots J_{s}^{(j)}|] \leq A(a/m)^s + AN^{-3}$ for each $s$. Firstly, we notice that the contribution of the $AN^{-3}$ terms in Inequality \ref{Jl decomposition} is negligible because there are $(r-2)N$ summands, and by Inequality \ref{mean value theorem bound}. We may also observe that $(a/m)\leq 1/2$ by our choices of $a$ and $m$ (for $N$ large enough). In particular, if $0<x<1/2$, then $\sum_{n=2}^{\infty}x^n \leq x^2/(1-x)\leq 2x^2$: this inequality follows from the sum of a geometric series. Combining these results, we see that the contribution of the $A(a/m)^s$ terms to the sum is also negligible. Thus, we have shown that the sum of the $D_1^{(s,j)}$ terms in Inequality \ref{Ji inequality} is negligible.

Next we consider the $D_2^{(j,s)}$ terms. The sum of the $D_2^{(j,s)}$ terms corresponds to $\varepsilon_0(t)$. Define $\xi_{j,s} := e^{it\mu_{j,s}}$ and $\gamma_{j,s} := \mathbb{E}[g_j(X_j)J_1^{(j)}\cdots J_s^{(j)}]\,\xi_{j,s}$ for each $j$ and $s$. Thus,
\begin{align*}
  \frac{1}{\sqrt{N}}\sum_{s=2}^{r-1}\sum_{j=1}^N\mathbb{E}\left[g_j(X_j)J_1^{(j)}\cdots J_{s}^{(j)}\right]D_2^{(s,j)}  &=\frac{1}{\sqrt{N}}\sum_{s=2}^{r-1}\sum_{j=1}^N\mathbb{E}\left[g_j(X_j)J_1^{(j)}\cdots J_{s}^{(j)}\right]\mathbb{E}\left[(\xi_{j,s} - \mathbb{E}\left[\xi_{j,s}\right])\,e^{itT/\sqrt{N}}\right]\\ &= \frac{1}{\sqrt{N}}\sum_{s=2}^{r-1}\mathbb{E}\left[\sum_{j=1}^N\left(\gamma_{j,s} - \mathbb{E}\left[\gamma_{j,s}\right]\right)e^{itT/\sqrt{N}}\right].\\
\end{align*} Let $\bar{\gamma}_{j,s}$ denote a centered copy of $\gamma_{j,s}$ (so  $\mathbb{E}[\bar{\gamma}_{j,s}] = 0$). By the H\"{o}lder inequality with coefficients 2 and 2, the triangle inequality, and the fact that $|e^{itT/\sqrt{N}}|=1$, we now calculate,
\begin{align}\left|\frac{1}{\sqrt{N}}\sum_{s=2}^{r-1}\sum_{j=1}^N \mathbb{E}\left[g_j(X_j)J_1^{(j)}\cdots J_{s}^{(j)}\right]D_2^{(j,s)}\right|^2&\leq\frac{1}{N}\sum_{s=2}^{r-1}\mathbb{E}\left[\sum_{j=1}^N\left(\gamma_{j,s} - \mathbb{E}\left[\gamma_{j,s}\right]\right)^2\right]\label{close terms}\\
&=\frac{1}{N}\sum_{s=2}^{r-1}\left(\sum_{1\leq j_1,j_2\leq N} \text{Cov}\,(\gamma_{j_1,s}, \gamma_{j_2,s})\right).\label{Covariance terms}
\end{align} 

Next, we decompose Sum \ref{Covariance terms} by considering the terms ($\text{Cov}\,(\gamma_{j_1,s}, \gamma_{j_2,s})$ for fixed $s$, $j_1$, and $j_2$) of this sum that are close to and far from each other separately. We consider first the terms in the sum such that  $|j_1 - j_2|<4ms$. From Inequality \ref{close terms}, and using the fact that $r\leq m$ for $N$ sufficiently large, we have, \begin{align}
    \frac{1}{N}\sum_{s=2}^{r-1}\left(\sum_{1\leq j_1,j_2\leq N, \, |j_1 - j_2|<4ms} \text{Cov}\,(\gamma_{j_1,s}, \gamma_{j_2,s})\right) &\leq A\frac{1}{N}\sum_{s=2}^{r-1}\sum_{j=1}^N \left((a/m)^s + N^{-3}\right)\,(4ms)\,(s^2 m^{-2})\\
    &\approx A\frac{1}{N}\sum_{s=2}^{r-1}\sum_{j=1}^N (a/m)^s m^{-1} s^3\\
    &= A\frac{1}{N}\sum_{s=2}^{r-1}s^3\sum_{j=1}^N (a/m)^{s+1}\\
    &\leq A\frac{1}{N}r^4 (a/m)^2\\
    &\leq A\frac{1}{N}m^4 (a/m)^2\\
    &\approx A\frac{m^2}{N}.
\end{align} For the first inequality, we grouped the terms by the distance $|j_1 - j_2|$ for each fixed $j_1$. The first approximation uses the fact that the $N^{-3}$ terms are negligible (by the same argument as for the $D_1^{(s,j)}$ terms) after taking the square root. The final approximation follows from our choices of $m$ and $r = C_p \log(N)$. Taking square roots, we conclude that the sum of these terms of Inequality \ref{close terms} such that  $|j_1 - j_2|<4ms$ is bounded by $AmN^{-1/2}$, and is thus negligible.

Next we consider the terms ($\text{Cov}\,(\gamma_{j_1,s}, \gamma_{j_2,s})$ for fixed $s$, $j_1$, and $j_2$) of Inequality \ref{far terms} such that  $|j_1 - j_2|\geq 4ms$. From Inequality \ref{far terms}, we have,
\begin{align}
\frac{1}{N}\sum_{s=2}^{r-1}\left(\sum_{1\leq j_1,j_2\leq N, \, |j_1 - j_2| \geq 4ms} \text{Cov}\left(\gamma_{j_1,s}, \gamma_{j_2,s}\right)\right)&\leq  \frac{1}{N} \sum_{s=2}^{r-1} \left(\sum_{j_1 = 1}^N \sum_{j_2=1}^N  \phi(4ms)\,2\left(A\left(a/m\right)^s+ AN^{-3}\right) \right)\label{far terms}\\
&\leq  \frac{A}{N} \sum_{s=2}^{r-1} \left(\sum_{j_1 = 1}^N \sum_{j_2=1}^N  \phi(4ms)\left(\left(a/m\right)^s+ N^{-3}\right) \right).
\end{align}
Inequality \ref{far terms} follows from the decoupling inequality, using the facts that $\mathbb{E}[\bar{\gamma}_{j_2,s}] = 0$, and that the indices are separated by at least $4ms$ indices (for each fixed $s$). The maximum follows from Inequality \ref{Ji inequality} and $|\xi_{j,s} - \mathbb{E}[\xi_{j,s}]|\leq 2$.

%%%%%%%%%%%%%%%%%%%
We have thus shown that $$\sum_{s=2}^{r-1}\sum_{j=1}^N\mathbb E\left[g(X_j)J_1^{(j)}\cdots J_s^{(j)}\right]\,\mathbb E\left[e^{itT'_{j,s+1}}\right]\approx 0.$$

It follows that, \begin{align}
    f'(t)\approx \sum_{j=1}^N\frac{1}{\sqrt{N}}\mathbb{E}\left[g_j(X_j)J_1^{(j)}\right]\mathbb{E}\left[e^{it(T/\sqrt{N}-\mu_{j,1})}\right].
\end{align}

%\textbf{NEED SOMETHING HERE- use (13) to show that $\sum_{s=2}^{r-1}\sum_{j=1}^N\mathbb E[g(X_j)J_1^{(j)}\cdots J_s^{(j)}]\mathbb E[e^{itT'_{j,s+1}}]\approx 0$.}

Now, we may recall by Inequality \ref{Ji inequality} (with $s=1$) that $\|g_j(X_j)J_1^{(j)}\|_{L^1}\leq Cm N^{-1/2}$. We also use the facts that $\mathbb{E}[|\mu_{j,1}|] \leq C\sqrt{m}N^{-1/2}|t|$ and $|e^{it(T/\sqrt{N}-\mu_{j,1})} - e^{itT/\sqrt{N}}|\leq |t\mu_{j,1}|$. Thus, by the linearity of variance of $S_N$, \begin{align*}
    \left|\sum_{j=1}^N\frac{1}{\sqrt{N}}\mathbb{E}\left[g_j(X_j)J_1^{(j)}\right]\left(\mathbb{E}[e^{it(T/\sqrt{N}-\mu_{j,1})}]-\mathbb{E}[e^{itT/\sqrt{N}}] \right)\right|&\leq \sum_{j=1}^N\frac{C\sqrt{m}}{\sqrt{N}}\frac{1}{N}|t|\,\mathbb{E}\left[|\mu_{j,1}|\right]\\
    &\leq \frac{Cm}{\sqrt{N}}.
\end{align*} Thus, recalling that $f(t) = \mathbb{E}[e^{itT/\sqrt{N}}]$, we have shown that,
\begin{align}
f'(t)\approx \frac{i}{\sqrt{N}}\left(\sum_{j=1}^N \mathbb{E}\left[g_j(X_j)\,J_1^{(j)}\right]\right)f(t).\label{O mN terms}   
\end{align} 

Furthermore, an application of the mean value theorem yields $J_1^{(j)} = it\Delta_{j,1} + O(mN^{-1})$. Summing the 
$O(mN^{-1})$ terms in Sum \ref{O mN terms}, we see that these terms are negligible. Thus, we have shown that, \begin{align}
f'(t)\approx \frac{i}{\sqrt{N}}\left(\sum_{j=1}^N \mathbb{E}\left[g_j(X_j)\,(it\Delta_{j,1})\right]\right)f(t) = \frac{-t}{\sqrt{N}}\left(\sum_{j=1}^N \mathbb{E}\left[g_j(X_j)\,\Delta_{j,1}\right]\right)f(t).
\end{align} 
Observe that for all $|j-k|\geq m = A_{p} N^{\epsilon_{p}}$, we have, \begin{align}
    \left|\mathbb{E}\left[g_j(X_j)\,g_k(X_k)\right]\right|&\leq \sqrt{N}\sqrt{N} \phi\left(|k-j|\right)\label{Using phi mixing third time}\\
    &\leq KN m^{-p} \\
    &\leq KN^{1-p\epsilon_p}\label{product bound}.
\end{align}

Furthermore, using the decomposition $ N^{-1/2} \sum_{k=1}^N g_k(X_k) =\Delta_{j,1} +S_{j,1}$ (valid for each $j$ and for $p$ large enough), we have,
\begin{align}
     \frac{1}{\sqrt{N}}\sum_{j=1}^N \mathbb{E}\left[g_j(X_j)\sum_{k=1}^N \frac{g_k(X_k)}{\sqrt{N}}\right] = O(N^{-1}) + \frac{1}{\sqrt{N}}\sum_{j=1}^N \mathbb{E}\left[g_j(X_j)\,\Delta_{j,1}\right].
\end{align}
This implies that,
\begin{align}
     f'(t) \approx \frac{-t}{\sqrt{N}}\left(\sum_{j=1}^N \mathbb{E}\left[g_j(X_j)\sum_{k=1}^N \frac{g(X_k)}{\sqrt{N}}\right]\right)f(t) \approx \frac{-t}{\sqrt{N}}\left(\sum_{j=1}^N \mathbb{E}\left[g_j(X_j)\,\Delta_{j,1}\right]\right) f(t).\label{Sjs decomposition}
\end{align}  The second approximation in \ref{Sjs decomposition} corresponds to $\varepsilon(t)$. We now consider the decomposition of the right-hand side of Approximation \ref{Sjs decomposition} into the terms close to and far from each index $k$ in the following way: \begin{align} \sum_{k=1}^N\mathbb{E}\left[g_k(X_k)\sum_{j:|j-k|\leq m} \frac{g_j(X_j)}{\sqrt{N}}\right] = \sum_{k=1}^N\mathbb{E}\left[g_k(X_k)\sum_{j=1}^N \frac{g_j(X_j)}{\sqrt{N}}\right] - \sum_{k=1}^N \mathbb{E}\left[g_k(X_k)\sum_{j:|j-k|>m}\frac{g_j(X_j)}{\sqrt{N}}\right].\label{near far decomposition}
\end{align}  By the linearity of variance of $S_N$, and using Inequality 
\ref{product bound}, the second term of the right-hand side of Equation \ref{near far decomposition} is of the order of $N^2\phi(m)\leq KN^2 N^{-p\epsilon_p}$. Thus, we have, \begin{align}
f'(t)\approx \frac{-t}{\sqrt{N}}\left(\sum_{j=1}^N \mathbb{E}\left[g_j(X_j)\,\Delta_{j,1}\right]\right) f(t) \approx \frac{-t}{N}\mathbb{E}\left[\left(\sum_{k=1}^Ng_k(X_k)\right)^2\right]f(t).\label{final approximation}
\end{align}
Finally, we assumed that $ \mathbb{E}\left[\left(\sum_{k=1}^Ng_k(X_k)\right)^2\right]=\mathbb{E}\left[S^2\right]  = N$. Combining this assumption with Approximation \ref{final approximation} then implies that $f'(t)\approx -tf(t)$, completing the proof. $\QEDB$\\

\section{Notes and Extensions of  Theorem \ref{Berry-Esseen Bounds for Statistics of Weakly Dependent Samples Proof}}
In the above theorem, the same asymptotic bounds as Bentkus et al. \cite{bentkus1997berry} were obtained while dropping the assumption that the family of random variables is stationary and replacing the exponential mixing rates with polynomial mixing rates, but with the additional assumption that the variance grows linearly.  

Notably, we did \textit{not} prove that optimal bounds were obtained (and likely should not be expected in general), either in the general setting of Theorem \ref{Berry-Esseen Bounds for Statistics of Weakly Dependent Samples Proof}, or in its applications.  However, the bounds we obtained are strong enough to be useful in many applications. It is not clear whether the $N^{\varepsilon_p}$ factors in the asymptotic bounds of Theorem \ref{Berry-Esseen Bounds for Statistics of Weakly Dependent Samples Proof} can be weakened (or removed) through a modification of the proof.  Since Theorem \ref{Berry-Esseen Bounds for Statistics of Weakly Dependent Samples Proof} does not ensure optimal rates in its applications, it is necessary to use other methods if optimal rates (or tighter) are required, without reference to Theorem \ref{Berry-Esseen Bounds for Statistics of Weakly Dependent Samples Proof}. In either case, future research could produce tighter bounds.

We argue that in certain applications of Theorem \ref{Berry-Esseen Bounds for Statistics of Weakly Dependent Samples Proof}, the assumption of linear variance growth can be dropped. To demonstrate this fact, we consider several applications of Theorem \ref{Berry-Esseen Bounds for Statistics of Weakly Dependent Samples Proof} to classes of random variables with non-linear variance growth. To extend this result to cases where the variance grows non-linearly, we consider a block decomposition of $S_N$. This is a commonly used technique in non-linear settings. For example, see  \cite{hafouta2021convergence}, \cite{hafouta2023almost}, and \cite{hafouta2021functional}, and Theorem \ref{Non Uniform Third Moment Application}.

\subsection{Extension of Theorem \ref{Berry-Esseen Bounds for Statistics of Weakly Dependent Samples Proof} to Polynomial Mixing Rates for Stationary Sequences} \label{Stationary extension section}
Theorem \ref{Berry-Esseen Bounds for Statistics of Weakly Dependent Samples Proof} immediately holds for faster mixing rates, such as exponential ($\phi(m)\leq C_0e^{-Cm}$) and stretched-exponential rates ($\phi(m)\leq C_0e^{-Cm^{\gamma}}$ for any $\gamma \in (0,1]$). This yields a strict generalization of Theorem 1.1 of Bentkus et al. \cite{bentkus1997berry} to polynomially fast mixing rates (rather than exponential rates) for \textit{stationary} sequences. This strict generalization comes from considering Theorem \ref{Berry-Esseen Bounds for Statistics of Weakly Dependent Samples Proof} for stationary sequences (linearity of variance follows from stationarity).

\section{Applications of Theorem \ref{Berry-Esseen Bounds for Statistics of Weakly Dependent Samples Proof}}\label{Applications of Main Theorem section} In this section, we apply Theorem \ref{Berry-Esseen Bounds for Statistics of Weakly Dependent Samples Proof} to obtain Berry-Esseen bounds for various classes of stochastic processes. In all of these applications, we choose $R_{j,k}$ for $j<k$ in a natural way: we let $R_{j,k}$ denote all of the remainder terms arising from the $X_{\ell}$ for all $\ell \in \{1,\cdots ,N\}\setminus [j,k]$. Furthermore, to apply Theorem \ref{Berry-Esseen Bounds for Statistics of Weakly Dependent Samples Proof} for each application, we  must estimate $\mathbb{E}[|R|]$ and $\gamma$, which can be done by estimating moments. The usefulness of Theorem \ref{Berry-Esseen Bounds for Statistics of Weakly Dependent Samples Proof} is that its assumptions are applicable to and easily verifiable in a wide class of processes. 

The applications of Theorem \ref{Berry-Esseen Bounds for Statistics of Weakly Dependent Samples Proof} to U-statistics and studentized sample means are generalizations of the applications in Bentkus et al. \cite{bentkus1997berry} to non-stationary, $\phi$-mixing sequences with linear variance and polynomial rates; their proofs are identical to those in \cite{bentkus1997berry}, as none of the proofs of applications in \cite{bentkus1997berry} rely upon the stationarity of the sequence $\{X_j\}_{j\in \mathbb{Z}^{+}}$ nor significantly upon the mixing rates. In particular,  all of the applications in Bentkus et al. \cite{bentkus1997berry} follow for non-stationary triangular arrays in which the variance grows linearly. However, proof is required to show that the assumption of linear variance can be dropped in any of these applications (\textit{if} variance linearization is possible). See Bentkus et al. \cite{bentkus1997berry} for proofs of Berry-Esseen theorems for a class of U-statistics, linear combinations of order statistics, functions of empirical distribution functions, studentized sample means, and functions of sample means under linear variance assumptions.

In the following subsection, we show that we can obtain Berry-Esseen bounds in general settings using Theorem \ref{Berry-Esseen Bounds for Statistics of Weakly Dependent Samples Proof}. The proof requires a block decomposition due to the non-linear variance growth and the non-uniform third moment.
%%%%%%%%%%%%%%%%%%%%%%%%%%%%%%%%%%%%%%%%%%%%%%%%%%%%%%%%%%%%%%%%%%%%%%%%%%%%%%%%%%%%%%%%%%%%%%%%%%%%%%%%%%%%%%%%%%%%%%%%%%%%%%%%%%%%%%%%%%%Start of Third moment Application
\subsection{A General Berry–Esseen Theorem for Non-Stationary Sequences With Non-Uniform Third Moments}\label{Non-uniform third moment subsection} 

%%%%%%%%%%%%%%%%%%%%%%%%%%%%%%%%%%%%%%%%%%%%%%%%%%%%%%%%%%%%%%%%%%%%%%%%%%%%%%%%%%%%%%%%%%%%%%%%%%%%%%%%%%%%End of Third Moment Application
In this application, we prove that we can obtain the same convergence rates as Theorem 2 for non-stationary, linear sequences $\{g_{j,N}(X_j)\}_{}$ \textit{without} assuming a uniform third moment bound.   We state the theorem for $R=0$, but the $R\neq 0$ case can also be handled by bounding $\mathbb{E}[R]$ and $\gamma$, as usual. 

In this proof, we fix a slowly growing $\tau = \tau(N)\to \infty$ (e.g. $\tau = \sigma_N^{\eta}$ for $0<\eta<2$) and define the exceptional set $$\mathcal{E}_\tau = \mathcal{E}_\tau(N):= \left\{j\leq N: \mathbb{E}\left[|g_j(X_j)|^3\right]>\tau\right\}.$$ The exact choice of $\tau$ does not affect the arguments in the proof of the theorem, but is important if we need to extend to $R\neq 0$ and when Assumption \ref{Inter block covariance control assumption} is not satisfied (see Corollary \ref{Non uniform third moment corollary}).

\textbf{Block Construction}: We employ the following two-step block construction.

\begin{enumerate}
    \item Remove $\mathcal{E}_\tau$ from $\{1,2, \cdots, N\}$ and form blocks $\tilde{Z}_j$ on the reduced  sequence by the following rule: form blocks by adding indices until the block variance is $\geq \tau^2$. The final block may be incomplete, but treat it as the final block.

    \item Re-insert the exceptional indices back into the nearby reduced blocks to create new blocks $Z_j$. 
\end{enumerate}

The block decomposition provides a partition $P_1, P_2, \cdots, P_{k_N}$ of $\{1,2, \cdots, N\}$. Denote $Z_j := \sum_{i\in P_j}g_i(X_i)$ and $Y_j:= Z_j/\tau$. In the proof, we demonstrate that this block decomposition both linearizes the variance and ensures bounded third moments for $\{Y_j\}$.

In Lemma \ref{Non-uniform third moment lemma}, we demonstrate that under Assumptions 1-6 of the theorem in this section, there exists $C_3<\infty$ (dependent on the choice of $\tau$) such that $\sup_{j,N}\mathbb{E}\,[|Y_j|^3]\leq C_3$. 

We make the following assumptions about the sequence $\{g_j(X_j)\}_{j=1}^{\infty}$:

\begin{enumerate}
    \item (Mixing) There exists $c_{\text{mix}} < \infty$ such that for all $h\in \mathbb{Z}^+$, $\phi(h)\leq c_{\text{mix}}\,h^{-p}$ for a large enough $p$.
    \item (Variance diverges) $\text{Var}(\sum_{j=1}^N g_j(X_j))\to \infty$.
    \item (Sub-linear variance growth) There exists $v_{\text{max}}>0$ such that for all $N$ and every subset $\mathcal{I} \subseteq \{1,2, \cdots, N\}$, $\text{Var}\left(\sum_{j\in \mathcal{I}} g_j(X_j)\right) \le v_{\text{max}}\, \text{Card}(\mathcal{I}).$ \label{Sub-linear variance growth assumption}
    \item (Average third moment) There exists $M<\infty$ such that for all $N$, $$\frac{1}{N}\sum_{j=1}^N\mathbb{E}\left[|g_j(X_j)|^3\right]\leq M<\infty.$$
    \item (Sparse exceptional mass) Assume that $$\sum_{i\in \mathcal{E}_\tau} \mathbb{E}\left[|g_i(X_i)|^3\right]\leq C_{\text{exc}}\,\tau^{3-\epsilon}$$ for constants $C_{\text{exc}},\epsilon>0$ independent of $N$ and $\tau$.
    \item $\max_{j}\|\max_{i\in\tilde P_j}|g_i(X_i)|\|_{L^3}\leq C\tau$ where $\tilde{P}_j := P_j \setminus \mathcal{E}_\tau$ and $C$ is a constant independent of $N$.
    
    \item\label{Inter block covariance control assumption} (Inter-block covariance control) Choose $C_3<\infty$ such that for all $N$, $\sup_{1\leq j \leq k_N} \mathbb{E}[|Y_j|^3]\leq C_3$ (this is possible from Assumptions 1-6, and $C_3$ can depend on the choice of $\tau$). Assume that the mixing constant coefficient $c_{\text{mix}}$ is small enough such that  $24(1/2)^{1/3} C_3 \,c_{\text{mix}}\leq 1$.
\end{enumerate}

Note that we replaced the uniform third moment assumption of Theorem 1 with an \textit{average} third moment assumption and inter-block covariance control. Larger choices of $\tau$ allow for a larger ``exceptional" third moment mass.

Also, recall a version of Rosenthal's Inequality for sufficiently fast mixing sequences. See \cite{PelBook}, Theorem 6.17 for a proof.

\begin{lemma}[Rosenthal's Inequality for $\phi$-mixing sequences]
 If the $\phi$-mixing coefficient of a real valued sequence $(A_j)_{j\geq 1}$ satisfy $\phi(n_0)<1/2$ for some $n_0$, then for all $p>2$ and $m<N$, 
$$\|A_m+A_{m+1}+\cdots+A_N\|_{L^p}\leq C_p\left(\left\|\max_{j=m}^N|A_j|\right\|_{L^p}+\max_{m\leq j\leq N}\|A_m+\cdots+A_{j}\|_{L^2}\right)$$
for some constant $C_p$ which does not depend on $N$ and $m$.   
\end{lemma}

\begin{lemma}\label{Non-uniform third moment lemma}
  Make Assumptions 1-6. There exists $C_3<\infty$ independent of $N$ such that $$\sup_{1\leq j \leq k_N} \mathbb{E}\left[|Y_j|^3\right]\leq C_3.$$  
\end{lemma}

\textbf{Proof}: We first observe that by Assumption 5, there are $O(1)$ exceptional indices (indeed, $C_{exc} \,\tau^{3-\epsilon} \geq \sum_{i\in \mathcal{E}_\tau} \mathbb{E}[|g_i(X_i)|^3]\geq \text{Card}(\mathcal{E}_\tau)\tau$). This implies that the block variances are still of order $\tau^2$, and $\text{Var}(Z_j)\asymp \tau^2$ after re-inserting vertices.

Next, we show that the blocks have uniform third moments. Define $\tilde{P}_j := P_j \setminus \mathcal{E}_\tau$ and $\tilde{Z}_j := \sum_{i\in \tilde{P}_j}g(X_j) $. To bound $\sup_{1\leq j \leq k_N} \mathbb{E}\left[|Y_j|^3\right],$ we apply Rosenthal's inequality. For fixed $j$\, this yields \begin{align}
    \mathbb{E}\left[\left|\tilde{Z}_j\right|^3\right]^{1/3}=\mathbb{E}\left[\left|\sum_{i\in \tilde{P}_j}g_i(X_j)\right|^3\right]^{1/3} &\leq C'\tau
\end{align} where $C'$ is independent of $N$. Thus, $\mathbb{E}[|\tilde{Z_j}|^3]\leq \tilde{C}_3\tau^{3}$. 

Thus, there exists a constant $C_3$ independent of $N$ such that \begin{align}
    \mathbb{E}\left[\left|Z_j\right|^3\right]\leq 4\left(\mathbb{E}\left[|\tilde{Z}_j|^3\right] + \mathbb{E}\left[\mathlarger{|}\sum_{i\in P_j \cap \mathcal{E}_\tau }g_i(X_i)\mathlarger{|}^3\right]\right) \leq 4\left(C' \tau^3+ C_{\text{exc}}\tau^{3-\epsilon}\right) \leq C_3\tau^3.
\end{align}

Thus, we can write \begin{align}
    \mathbb{E}\left[|Y_j|^3\right] \asymp \mathbb{E}\left[|Z_j/\tau|^3\right] \leq C_3.
\end{align} 
\QEDB

This proves that Assumption 7 is justified under Assumptions 1-6. Now we prove the main theorem of this section.

\begin{theorem}[Berry-Esseen Bounds without Uniform Third Moment Bound]\label{Non Uniform Third Moment Application}
    Make Assumptions 1-7, and define $T  =  \sum_{j=1}^N g_j(X_j)$ (note $R=0$). There exists a constant $C<\infty$ independent of $N$ and $\varepsilon_p\to 0$ such that
\begin{align*}
    \sup_{x\in \mathbb{R}}\left|\bbP(T<\sigma_{N}x) -\Phi(x)\right|\leq C\sigma_N^{-1 + \varepsilon_p}.
\end{align*}
\end{theorem}

\textbf{Proof}: By Lemma \ref{Non-uniform third moment lemma}, we know that $\sup_{1\leq j \leq k_N} \mathbb{E}[|Y_j|^3]\leq C_3$ and block sizes are of order $\tau^2.$ We will show that the block decomposition provides a partition $P_1, P_2, \cdots, P_{k_N}$ of $\{1,2, \cdots, N\}$ that satisfies $k_N \asymp \sigma_N^2$ and $\text{Var}(\sum_{i\in P_j} g_i(X_i))\asymp \tau^2$ for each $j$.

We next ensure that $\text{Var}(\sum_{j=1}^{k_N} Y_j)\asymp k_N$. By applying Corollary \ref{Rio Covariance Corollary}
(using that $\alpha(h)\leq \phi(h)/2$) we see that \begin{align}
    2\,\mathlarger{|}\sum_{1\leq i < j \leq k_N}\text{Cov}(Y_i, Y_j)\mathlarger{|}
    &\leq 2\,\sum_{j=1}^{k_N-1} 6\, \text{max}\left(\mathbb{E}\left[|Y_i|^3\right]\right)^{2/3}\,j\,\alpha(j)\\
    &\leq 12 (1/2)^{1/3} C_3 \,c_{\text{mix}}\, k_N.\label{inter block covariance bound}
\end{align} Recall that \begin{align}
    \text{Var}\left(\sum_{j=1}^{k_N} Y_j\right) = \sum_{j=1}^{k_N} \text{Var}(Y_j) + 2\sum_{1\leq i < j \leq k_N} \text{Cov}(Y_i, Y_j).
\end{align} In particular, we have linear variance over the blocks since, $\sum_{j=1}^{k_N} \text{Var}(Y_j) \geq 4\,|\sum_{1\leq i < j \leq k_N}\text{Cov}(Y_i, Y_j)|$ and $\text{Var}(Y_j)\geq1 $. 
Thus, we may apply Theorem 2 with $R=0$, yielding the result. 
\QEDB

\begin{remark}
    We could have replaced Assumptions 3 and 4 with other assumptions. In particular, Assumptions 3 and 4 ensure that the exceptional third-moment contribution for each block is controlled, and does not outweigh the non-exceptional contribution. The benefit of Assumptions 3 and 4 is that they are easily verifiable in practice. Assumption 5 gives a concrete criterion that the variance of the block sums grows linearly in terms of only the mixing coefficients. While it is a natural assumption, it is entails a non-trivial additional requirement. 
    
    As for Assumption 6, it clearly holds if $\|g_j\|_\infty,\, j\not\in\mathcal E_\tau$ is of size $\tau$, which can be applied when $\tau=\sigma_N^a$ for some a small enough and $\|g_j\|_{\infty}\asymp j^a,\, j\not\in\mathcal E_\tau$, for instance. The unbounded case can be treated by truncation. One can replace $g_j$ with $g_j\mathbbm 1_{\{|g_j|\leq \sigma_N^a\}}$. However, the issue here is that we sum of the error bounds so that this will yield something helpful only when $\sigma_N\geq cN^{\epsilon}$ for an appropriate (small) $\epsilon>0$. The exact details are left for the reader.
\end{remark}

In the following corollary, we show that even without Assumption 7, Berry-Esseen rates can still be obtained  by inserting small  gaps (relative to the blocks) between the blocks. The gaps ensure that $\text{Cov}(Y_i, Y_{i+1})$ is small for all $i$ via mixing. The proof of this corollary involves bounding $\mathbb{E}\,[|R|]$ and $\gamma$.

\begin{corollary}\label{Non uniform third moment corollary}
    Make all of the assumptions of the theorem except Assumption 7. Then, for any choices $0<\beta<\alpha < 2$, there exists $C_{\alpha,\beta}>0$ independent of $N$ such that for $p$ large enough, \begin{align*}
    \sup_{x\in \mathbb{R}}\left|\bbP(T<\sigma_N  x ) - \Phi(x)\right|\leq  C_{\alpha,\beta}\, \sigma_N^{2\beta/3 - 5\alpha/6 + 2/3 + \varepsilon_p} + C_{\alpha,\beta}\,\sigma_N^{-1+\varepsilon_p}.
\end{align*}
\end{corollary}

\textbf{Proof}: Construct alternating blocks $Y_i = Z_i/\tau$ and gaps $G_i/\tau$ by adding indices until a variance threshold for each is met; this ensures that we have $\text{Var}(Z_i) \asymp \tau^2 \asymp \sigma_N^{\alpha}$ and $\text{Var}(G_i) \asymp \sigma_N^\beta$ for any fixed $0<\beta < \alpha < 2$  for all $i, N$. We let $R = \sum_{j=1}^{k_N} G_j/\tau$ and $S = \sum_{j=1}^{k_N} Y_j$, and we must now bound $\mathbb{E}[|R|]$ and $\gamma$.

First we bound $\mathbb{E}[|R|]$. We write \begin{align*}
    \mathbb{E}\left[|R|\right] &= \frac{1}{\tau}\mathbb{E}\left[\mathlarger|\sum_{j=1}^{k_N} G_j\mathlarger|\right]\\
    &\leq  \frac{1}{\tau} \sum_{j=1}^{k_N} \sqrt{\text{Var}(G_j)}\\
    &\leq \frac{1}{\tau} k_N \sqrt{C\sigma_N^\beta}\\
    &\asymp \sigma_N^{\beta/2 - \alpha/2} \cdot \frac{\sigma_N^{2}}{\sigma_N^{\alpha}} \\
    &= \sigma_N^{2-3\alpha/2 + \beta}.
\end{align*}

Now we bound $\gamma$. We will be bounding
$$\mathbb{E}\left[|D_k|^2\right]
:= \mathbb{E}\left[|R_{j,k}-R_{j,k-1}|^2\right] 
= \mathbb{E}\left[\mathlarger|\mathbb{E}\left[R\,|\,\sigma^C[j,k]\right]
-\mathbb{E}\left[R\,|\,\sigma^C[j,k-1]\right]\mathlarger|^2\right],$$
and then writing \begin{align}
    \mathbb{E}\left[R\,|\,\sigma^C[j,k]\right] &= \frac{1}{\tau}\mathbb{E}\left[\left(\sum_{i=j}^k G_k\right)+ \text{other }G_i \text{ terms } |\,\sigma^C[j,k]\right] \\
    &= \frac{1}{\tau}\left[\mathbb{E}\left[\left(\sum_{i=j}^k G_k\right) \,|\,\sigma^C[j,k]\right] + \text{other }G_i \text{ terms}\right].
\end{align}
Thus, after cancellation of terms under subtraction,\begin{align}
    \mathbb{E}\left[R\,|\,\sigma^C[j,k]\right]-\mathbb{E}\left[R\,|\,\sigma^C[j,k-1]\right] = \frac{1}{\tau}\mathbb{E}\left[\sum_{i=j}^k G_k\,|\, \sigma^C[j,k]\right] - \frac{1}{\tau}\mathbb{E}\left[\sum_{i=j}^k G_k\,|\, \sigma^C[j,k-1]\right].
\end{align} Using the real inequality $(x+y)^2\leq 2x^2 + 2y^2$ we can then write \begin{align}
    \text{Var}(D_k) &\leq 2\,\text{Var}\left(\mathbb{E}\,\mathlarger[\sum_{i=j}^k G_i\,|\, \sigma^C[j,k]\mathlarger]\right) + 2\,\text{Var}\left(\mathbb{E}\,\mathlarger[\sum_{i=j}^k G_i\,|\, \sigma^C[j,k-1]\mathlarger]\right)\\
    &= 2\,\text{Var}\left(\mathbb{E}\,\mathlarger[\sum_{i=j}^k G_i\,|\, \sigma^C[j,k]\mathlarger]\right) + 2\,\text{Var}\left(\mathbb{E}\,\mathlarger[G_k + \sum_{i=j}^{k-1} G_i\,|\, \sigma^C[j,k-1]\mathlarger]\right)\\
    &\leq 2\text{Var}\left(\mathbb{E}\mathlarger[\sum_{i=j}^k G_i\,|\, \sigma^C[j,k]\mathlarger]\right) + 4\text{Var}\left(\mathbb{E}\mathlarger[\sum_{i=j}^{k-1} G_i\,|\, \sigma^C[j,k-1]\mathlarger]\right) + 4\text{Var}\left(\mathbb{E}\mathlarger[G_k\,|\, \sigma^C[j,k-1]\mathlarger]\right).\label{Var Dk decomposition}
\end{align} The first two terms of \ref{Var Dk decomposition} will be controlled through mixing, the third through a crude bound. 

First we bound  $\text{Var}\,(\mathbb{E}\,[\sum_{i=j}^k G_i \,|\, \sigma^C[j,k]])$. We write \begin{align}
    \text{Var}\left(\mathbb{E}\,\mathlarger[\sum_{i=j}^k G_i\,|\, \sigma^C[j,k]\mathlarger]\right) &= \sup_{g\in L^2,\, \|g \|_{L^2 (\sigma^C[j,k])} = 1}\left[\text{Cov}\,\mathlarger{(}\sum_{i=j}^k G_i, g\mathlarger{)}\right]^2\\
    &\leq \sup_{g\in L^2,\, \|g \|_{L^2 (\sigma^C[j,k])} = 1}\left[\phi(r) \left\|\sum_{i=j}^k G_i\right\|_{L^2}\right]^2\\
    &\leq C\left[N^{-\beta p}\cdot o(\sigma_N^2)]^2\right]^2,
\end{align} which is trivial for large $p$ by Assumption \ref{Sub-linear variance growth assumption}. The same idea works in bounding $4\,\text{Var}(\mathbb{E}\,[\sum_{i=j}^{k-1} G_i\,|\, \sigma^C[j,k-1]])$. 

Next, we bound $\text{Var}(\mathbb{E}\,[G_k\,|\, \sigma^C[j,k-1]])$. Since each gap $G_k$ is of variance $\text{Var}(G_k)\asymp \sigma_N^{\beta}$, we have that \begin{align}
    \text{Var}\left(\mathbb{E}\left[G_k\,|\, \sigma^C[j,k-1]\right]\right)\leq \text{Var}\left(G_k\right) = \text{Var}\left( \sum_{i\in G_k} \frac{g(X_i)}{\tau}\right) = O\left(\frac{r}{\tau^2}\right) = O\left(\sigma_N^{\beta-\alpha}\right).
\end{align}

Thus, using our previous bounds, as well as $\|D_k\|_{L^1} \leq 2\,\mathbb{E}[|R|]$, and assuming that $p$ is large enough that $\text{Var}(G_k)$ is the leading contribution of $\text{Var}(D_k)$, we have that
\begin{align}
    \mathbb{E}[|D_k|^{3/2}]^{2/3} &\leq \|D_k\|_{L^2}^{2/3}\, \| D_k\|_{L^1}^{1/3}\\
    &\leq C\,[\sigma_N^{\beta-\alpha}]^{1/3}\,[\sigma_N^{2-3\alpha/2 + \beta}]^{1/3}\\
    &= C\,\sigma_N^{2\beta/3 - 5\alpha/6 + 2/3}, \label{Gamma final bound}
\end{align}
where $C$ is a constant independent of $j$ and $N.$ Applying Theorem 1 with these bounds for $\mathbb{E}[|R|]$ and $\gamma$ completes the proof.
\QEDB

For this corollary, it is ideal that we optimize the Berry-Esseen rate by choosing $\beta$ small and $\alpha$ large. Doing so we may obtain non-trivial rates. It is likely possible to obtain faster rates in Corollary \ref{Non uniform third moment corollary} by modifying our proof.

If we were to include an external remainder $R\neq 0$, then we would need to bound $\mathbb{E}[R]$ and $\gamma$, defined with respect to the  sequence $\{Y_j\}_{j=1}^{k_N}$.

\subsection{Application to Uniformly Bounded Summands Without Assumptions on the Variance Growth Rate}\label{Bounded}
Let us consider a sequence of the form 
$$
T_N=\sum_{j=1}^Ng_j(X_j)+R_N(X_1, \cdots,X_N):=S_N+R.
$$
Assume that $\sup_{j,N}\|g_j(X_j)\|_{L^\infty}<\infty$ and that $\bbE[g_j(X_j)]=0$ for all $j,N$. Then by \cite[Theorem 6.17]{PelBook} there exists a constant $C>0$ such that for every $j<k$ we have 
$$
\left\|\sum_{\ell=j}^{j+k-1}g_\ell(X_\ell)\right\|_{L^3}\leq C\left(1+\max_{L\leq k}\left\|\sum_{\ell=j}^{j+L-1}g_\ell(X_\ell)\right\|_{L^2}\right).
$$

Next, let us assume that $\sigma_N=\|S_N\|_{L^2}\to\infty$. In what follows we explain how to prove CLT rates without any other growth rates on $\sig_n$. Let us take a sufficiently large constant $A>1$, let us set $a_1=1$ and let $b_1$ be the first index such that $\|S_{b_1}\|_{L^2}\geq A$. Next, set $a_2=b_1+1$ and let $b_2>a_1$ be the first index such that $\|S_{b_2}-S_{b_1}\|_{L^2}\geq A$. Continuing this way, we get intervals in the integers $B_{i,N}=\{a_i,a_{i}+1,\cdots,b_i\}$ such that $a_i<b_i<a_{i+1}$ and their union is $\{1,2,\cdots,N\}$. 
Define 
$$
\tilde X_j=\tilde X_{j,N}:=(X_k)_{k\in B_{j,N}}
$$
and 
$$
G_j(\tilde X_j):=\sum_{k\in B_{j,N}}g_k(X_k).
$$
Then 
$$
\sup_{j}\|G_j(\tilde X_j)\|_{L^3}<\infty.
$$
Now, since $\phi(m)=O(m^{-p})$ and $p$ is large enough, we have
$$
 \sup_k\sum_{j\not=k}|\text{Cov}\,(G_k(\tilde X_j),G_j(\tilde X_j))|\leq \sup_k\sum_{j\not\in B_k}|\text{Cov}\,(G_k(\tilde X_k),g_j(X_j))|\leq BA
$$
for some constant $B$ which does not depend on $A$. Here we used that  $\rho(m)\leq 2\sqrt{\phi(m)}=O(m^{-p/2})$ and that 
$$
|\text{Cov}(G_k(\tilde X_k),g_j(X_j))|\leq 4\,\rho(m_{j,k})\|G_k(\tilde X_k)\|_{L^2}\|g_j(X_j)\|_{L^2},
$$
where $m_{j,k}$ is the distance between $j$ and $B_k$ (so what we need is $p>2$ here).
Therefore, if $A$ is large enough, then
 for every interval in the integers $I\subset \{1,2,\cdots,k_N\}$ we have 
$$
a\,\text{Card} (I)\leq\text{Var}\left(\sum_{j\in I}G_j(\tilde X_j)\right)\leq b\, \text{Card} (I),
$$
where $a,b>0$ are constants that depend only on $A$, and $\text{Card}(I)$ is the size of $I$. It also follows that the number $k_N$ of sets $B_{i,N}$ satisfies $k_N\asymp\sig_N^2$.

Next, write
$$
T=\sum_{j=1}^{k_N}G_j(\tilde X_j)+\tilde R,
$$
where 
$$
\tilde R=R=R(\tilde X_1,\tilde X_2,\cdots,\tilde X_{k_N}).
$$
Thus, the entire setup reduces to the case when the variance of the partial sum is linear in the number of summands $k_N$ and the third moment is uniformly bounded. We thus get CLT rates of order $\sig_N^{-1+\varepsilon_p}$ under appropriate restrictions on $\tilde\gamma$ and $\bbE[|R|]$, where $\tilde\gamma$ corresponds to $\gamma$ when viewing $R$ as a function of $\tilde X_j$.

%%%%%%%%%%%%%%%%%%%%%%%%%%%%%%%%%%%%%%%%%%%%%%%%%%%%%%%%%%%%%%%%%%%%%%%%%%%%%%%%%%%%%%%%%%%%%%%%%%%%%%%%%%%%%%%%%%%%%%%%%%%%%%%%%%%%%%%%%%%%%%%%%%%%%%%%%%%%%%%%%%%%%%%%%%%%%%%%%%%%%%%%%%%%%%%%%%%%%%%%%%%%%%%%
\subsection{Alternative Bounds for Remainder Terms}
The following theorem can be used to easily obtain Berry-Esseen rates for statistics of the form $T' = T + R'$, where $T = S+R$ is a statistic with Berry-Esseen rates already established. The following theorem is useful when it is difficult to bound $\gamma$, but we have control of the variance of $R'_N$, where $R'_N$ is a (possibly additional) remainder term. This requires some level of control over the size of $R'_N$, as in Theorem \ref{Berry-Esseen Bounds for Statistics of Weakly Dependent Samples Proof}. We could, for example, assume that $\text{Var}(R'_N) = o(\sigma_N^2)$ and $\mathbb{E}\,[|R'_N|]= o(\sigma_N)$. These are both very natural assumptions and seem to be minimal. To obtain Berry-Esseen bounds using Theorem \ref{Berry-Esseen Bounds for Statistics of Weakly Dependent Samples Proof}, we must bound $\mathbb{E}[R_N]$ and $\gamma$.  The following theorem provides a simple alternative for obtaining Berry-Esseen rates using Theorem \ref{Non Uniform Third Moment Application}, with a remainder term or an additional remainder term. 

\begin{theorem}[Alternative Remainder Bounds for Theorem \ref{Non Uniform Third Moment Application} for Statistics with Additional Remainder] \label{Remainder Theorem}
\end{theorem} \textit{Let $T_N = S_N + R_N = \sum_{j=1}^N g_j(X_j) + R_N$, and let $R'_N$ be a remainder term satisfying $\text{Var}\,(R'_N) = o(\sigma_N^2)$. Assume that $$\sup_{x\in\mathbb{R}}|\bbP(T< \sigma_N x) - \Phi(x)| \leq r_N,$$ where $r_N$ is a Berry-Esseen rate that has already been established (for example, by Theorem \ref{Berry-Esseen Bounds for Statistics of Weakly Dependent Samples Proof}). Define $T'_N := T_N + R'_N$, $(\sigma'_N)^2 := \text{Var}\,(T'_N)$, and $F'_N(x):= \bbP(T'_N/\sigma'_N\leq x)$. Then, for any $\delta \in (0,1)$, there exists a sufficiently large $C>0$ independent of $N$ such that,  \begin{align}
    \sup_{x\in\mathbb{R}}|F_N'(x) - \Phi(x)| \leq 3r_N +C\left|\frac{\sigma'_N}{\sigma_N} - 1\right|^{1-\delta} +   C\left(\frac{\text{Var}\,(R'_N)}{\sigma_N^2}\right)^{1/3} =o(1).\label{Perturbed Berry Esseen Bounds}
\end{align}} 

\textbf{Proof}: We begin with the calculation \begin{align}
    \text{Var}\left(T'_N\right) = \text{Var}\left(\sum_{j=1}^N g_j(X_j)\right) + \text{Var}\left(R'_N\right)+2\,\text{Cov}\left(\sum_{j=1}^N g_j(X_j), R'_N\right).
\end{align} By the H\"older inequality,
\begin{align}
    \text{Cov}\left(\sum_{j=1}^N g_j(X_j), R'_N\right)\leq \sqrt{\text{Var}\left(\sum_{j=1}^N g_j(X_j)\right) \,\text{Var}\left(R'_N\right)} = o(\sigma_N^2).
\end{align} Furthermore  we have shown that $\sigma_N'/\sigma_N \to 1$, since we showed that $(\sigma_N')^2 = \sigma_N^2(1+o(1))$. Thus, we have ensured the the variance of $T'_N$ is of order $\sigma_N^2$. 

 However, we are now considering a system ``perturbed'' by $R'_N$.  Denote $F'_N (x) = \bbP(T'_N/\sigma'_N\leq x)$ and $F_N(x) = \bbP(T_N/\sigma_N \leq x)$. Note that $$F_N'(x)=\bbP\left(\frac{T_N'}{\sigma'_N}\leq x\right) = \bbP\left(\frac{T_N + R'_N}{\sigma_N}\leq \frac{\sigma'_N}{\sigma_N}x - \frac{R'_N}{\sigma_N}\right)= \bbP\left(\frac{T_N}{\sigma_N}\leq \frac{\sigma_N'}{\sigma_N}x-\frac{R'_N}{\sigma_N}\right) = F_N\left(\frac{\sigma_N'}{\sigma_N}x-\frac{R'_N}{\sigma_N}\right).$$

Thus, we may write 
\begin{align}
    \left|F'_N(x)- \Phi(x)\right| \leq \left|\Phi(\frac{\sigma'_N}{\sigma_N}x) - \Phi(x)\right|+ \left|F_N(\frac{\sigma'_N}{\sigma_N}x) - \Phi(\frac{\sigma'_N}{\sigma_N}x)\right| + \left|F_N(\frac{\sigma'_N}{\sigma_N}x) - F_N(\frac{\sigma'_N}{\sigma_N}x - \frac{R'_N}{\sigma_N})\right|.\label{CDF decomposition}
\end{align} 
Now, we must bound each term of \ref{CDF decomposition}. We already know that $|F_N(\frac{\sigma'_N}{\sigma_N}x) - \Phi(\frac{\sigma'_N}{\sigma_N}x)|\leq r_N$.

Next, we bound $\sup_{x\in\mathbb{R}}|\Phi(\frac{\sigma'_N}{\sigma_N}x) - \Phi(x)|$. To do so, we let $M_N$ be a sequence of positive numbers such that $M_N \to \infty$, and consider separately the cases $|x|\leq M_N$ and $|x|>M_N$. For $|x|\leq M_N$, we have $$\left|\Phi(\frac{\sigma'_N}{\sigma_N}x) - \Phi(x)\right| \leq C \left|\frac{\sigma'_N}{\sigma_N}x - x\right|= C \,|x| \left|\frac{\sigma'_N}{\sigma_N} - 1\right| \leq C M_N \left|\frac{\sigma'_N}{\sigma_N} - 1\right|.$$ Choosing $M_n = |\frac{\sigma'_N}{\sigma_N} - 1|^{-\delta}$ for some $\delta \in (0,1)$, we have that $$|\Phi(\frac{\sigma'_N}{\sigma_N}x) - \Phi(x)|\leq C\,|\frac{\sigma'_N}{\sigma_N} - 1|^{1-\delta} = o(1).$$ Now consider the case $|x|>M_N$. Denote by $Z$ a standard normal random variable. If $x>M_N$, then $|\Phi(\frac{\sigma'_N}{\sigma_N}x) - \Phi(x)| \leq \bbP(Z>M_N) \to 0$ exponentially fast in $M_N$ as $N\to \infty$ by Mill's ratio and because $M_N\to \infty$ as $N\to \infty$. The case $x<-M_N$ is the same. We have thus shown that $\sup_{x\in\mathbb{R}}|\Phi(\frac{\sigma'_N}{\sigma_N}x) - \Phi(x)| \leq C\,|\frac{\sigma'_N}{\sigma_N} - 1|^{1-\delta}  = o(1)$.

Finally, we bound $\sup_{x\in \mathbb{R}}|F_N(\frac{\sigma'_N}{\sigma_N}x) - F_N(\frac{\sigma'_N}{\sigma_N}x - \frac{R'_N}{\sigma_N})|$. We denote $y = \frac{\sigma'_N}{\sigma_N}x$ and $u = \frac{R_N}{\sigma_N}$.

We now have the decomposition, valid for each $\varepsilon_N>0$ (we will later determine the appropriate choice of $\varepsilon_N$), \begin{align}
    \bbP\left(\frac{T_N}{\sigma_N} \leq y-u\right) &= \mathbb{E}\left[\mathbbm{1}_{\{|u|\leq \varepsilon_N\}} \mathbbm{1}_{\{\frac{T_N}{\sigma_N} \leq y-u  \}}\right] + \mathbb{E}\left[\mathbbm{1}_{\{|u| > \varepsilon_N\}} \mathbbm{1}_{\{\frac{T_N}{\sigma_N} \leq y-u  \}}\right].
\end{align} We bound both terms separately.

We first have that $$\mathbb{E}\left[\mathbbm{1}_{\{|u| > \varepsilon_N\}} \mathbbm{1}_{\{\frac{T_N}{\sigma_N} \leq y-u  \}}\right]\leq \bbP\left(|u| > \varepsilon_N\right) = \bbP\left(\frac{|R_N|}{\sigma_N} > \varepsilon_N\right) \leq \frac{\mathbb{E}\,[|R_N|^2]}{\sigma_N^2 \varepsilon_N^2},$$ where the last inequality follows from Chebyshev's inequality.

Next, we bound $\mathbb{E}[\mathbbm{1}_{\{|u| \leq \varepsilon_N\}} \mathbbm{1}_{\{\frac{T_N}{\sigma_N} \leq y-u  \}}]$. Since $|u|\leq \varepsilon_N$, we know that $y-u\in [y-\varepsilon_N, y+ \varepsilon_N]$. Thus, using that $\mathbbm{1}_{\{\frac{T_N}{\sigma_N} \leq y-u \} }$ is monotone in $u$,
$$
F_N(y-\epsilon_N)-\mathbb P(|u|>\epsilon_N)\leq \mathbb{E}\left[\mathbbm{1}_{\{|u| \leq \varepsilon_N\}} \mathbbm{1}_{\{\frac{T_N}{\sigma_N} \leq y-u  \}}\right]\leq F_N(y+\epsilon_N).
$$

We have thus shown that $\bbP(\frac{T_N}{\sigma_N}\leq y-u)\leq \sup_{|v| \leq \varepsilon_N}F_N(y-v)+ \bbP(|u| > \varepsilon_N)$. Subtracting $F_N(y)$, taking the supremum over $y$, and then applying the triangle inequality, we have $$\sup_{y\in \mathbb{R}} \left|\bbP(\frac{T_N}{\sigma_N} \leq y-u) - F_N(y)\right| \leq \sup_{y\in \mathbb{R}} \left|\sup_{|v|\leq \varepsilon_N} F_N(y-v) - F_N(y)\right| + 2\,\bbP\left(|u|>\varepsilon_N\right).$$ By the triangle inequality, for each $v$ we can write 
$$
\left|F_N (y-v) - F_N(y)\right| \leq |F_N(y) - \Phi(y)| + |\Phi(y) - \Phi(y-v)| + |F_N(y-v) - \Phi(y-v)|.$$ 
The terms $|F_N(y) - \Phi(y)|$ and $|F_N(y-v) - \Phi(y-v)|$ are both bounded by $r_N$. By the normal density, $|\Phi(y) - \Phi(y-v)| \leq \frac{1}{\sqrt{2\pi}} \varepsilon_N$. Thus, we have shown that \begin{align}
    \sup_{y\in \mathbb{R}}\left|\bbP(\frac{T_N}{\sigma_N} \leq y-u) - F_N(y)\right|&\leq 2r_N + \frac{1}{\sqrt{2\pi}}\varepsilon_N + 2\,\bbP\left(|u|>\varepsilon_N\right)\\
    &\leq 2r_N + \frac{1}{\sqrt{2\pi}}\varepsilon_N + \frac{2\,\mathbb{E}\,[|R_N|^2]}{\sigma_N^2\varepsilon_N^2}.\label{perturbation bound}
\end{align}

We wish to balance the terms in  Inequality \ref{perturbation bound} with $\varepsilon_N$. That is, we want $\frac{1}{\sqrt{2\pi}}\varepsilon_N \asymp \frac{\mathbb{E}[|R_N|^2]}{\sigma_N^2\varepsilon_N^2}$. Solving this for $\varepsilon_N$, this is equivalent to $$\varepsilon_N\asymp (\frac{\text{Var}(R'_N)}{\sigma_N^2})^{1/3}.$$

With this choice of $\varepsilon_N$, we have proven the bound $\sup_{y\in\mathbb{R}} |F_N(y-u) - F_N(y)| \leq 2r_N + O((\frac{\text{Var}(R'_N)}{\sigma_N^2})^{1/3})$. This gives a final bound of 

\begin{align*}
    \sup_{x\in\mathbb{R}}|F_N'(x) - \Phi(x)| &\leq \left[O(|\frac{\sigma'_N}{\sigma_N} - 1|^{1-\delta})\right] + [r_N] + \left[2r_N + O((\frac{\text{Var}(R'_N)}{\sigma_N^2})^{1/3})\right] \\
    &= 3r_N +O\left(|\frac{\sigma'_N}{\sigma_N} - 1|^{1-\delta}\right) +   O\left((\frac{\text{Var}(R'_N)}{\sigma_N^2})^{1/3}\right) =o(1).
\end{align*} By our assumptions that $\text{Var}(R_N') = o(\sigma_N^2)$ both $O((\frac{\text{Var}(R'_N)}{\sigma_N^2})^{1/3})]$ and $O(|\frac{\sigma'_N}{\sigma_N} - 1|^{1-\delta})$ are $o(1)$. 
$\QEDB$

Note that $O((\frac{\text{Var}(R'_N)}{\sigma_N^2})^{1/3})$ and $O(|\frac{\sigma'_N}{\sigma_N} - 1|^{1-\delta})$  were left in the final Berry-Esseen bound \ref{Perturbed Berry Esseen Bounds}. Without additional information about $R'_N$, however, the best rate we are guaranteed will be $o(1)$. However, these terms account for the case that we have a stronger rate of convergence than $\text{Var}(R'_N) = o(1)$. If, say, $\text{Var} (R'_N) \asymp O(N^{-\xi})$ for $\xi \in (0,1)$, then $O(\frac{\text{Var}(R'_N)}{\sigma_N^2})^{1/3} = O(N^{-\xi/3})$ and $O(|\frac{\sigma'_N}{\sigma_N} - 1|^{1-\delta}) = O(N^{-\xi(1-\delta)/2})$ assuming $\sigma_N^2\asymp N$. If we also assumed higher moment assumptions for $R'_N$, then we could modify the usage of Chebyschev's inequality in the proof to obtain sharper bounds.

This application of Theorem \ref{Berry-Esseen Bounds for Statistics of Weakly Dependent Samples Proof} is useful, as it allows for an extra added error term once we have already obtained Berry-Esseen rates for a given process.

This theorem is also useful in non-linear settings where a block decomposition is used along with gaps inserted between the blocks. It provides a simple avenue to bound the contribution of the gaps.

\subsection{U-Statistics} The application in this section is proved in Bentkus et al. \cite{bentkus1997berry} for \textit{stationary} U-statistics. The following theorem is an analogue of Theorem 2.5 in \cite{bentkus1997berry} for non-stationary sequences with non-linear variance growth. Our proof is a modified version of the proof in \cite{bentkus1997berry}, which proved this theorem for stationary sequences and exponential mixing rates. Under stationarity and exponential mixing, we would obtain the same rates for non-stationary sequences with linear variance growth directly from their proof and the exponential mixing analogue of Theorem \ref{Berry-Esseen Bounds for Statistics of Weakly Dependent Samples Proof}.

Let $\{X_i\}_{i\in \mathbb{Z}^{+}}$ be a non-stationary $\phi$-mixing sequence. For simplicity, we consider U-statistics of second order only (it may be possible to consider higher order U-statistics by modifying our proof, but the proof would become significantly more computationally involved).  Define the U-statistic, $$U' := S'+R' \text{, where } S' = \sum_{j=1}^{N}g_j(X_j) \text{ and } R' = \frac{1}{N}\sum_{1\leq i < j \leq N} \psi_{i,j}\left(X_i, X_j\right),$$  where $g_j$ and $\psi = \psi_{i,j}$ may depend on $N$\footnote{U-statistics (``U" stands for unbiased) are a class of nonparametric statistics (that is, statistics that do not make assumptions about the sampling distribution) that can be represented in the following form, assuming that the function $\psi$ is symmetric (under permutations of indices): $$f_N(x_1, \cdots, x_N) = \frac{1}{{N \choose r}}\sum_{1\leq i_1<\cdots< i_r\leq N}\psi\left(x_{i_1}, \cdots, x_{i_r}\right).$$  The function $\psi$ is called the kernel of the U-statistic, and can be chosen arbitrarily (although most limit theorems make some assumptions about $\psi$). We can observe that this aligns with our construction, for $r=2$. There is a more general form of U-statistics in which we do not assume that $\psi$ is symmetric, but we do not use this more general form. 

 U-statistics were introduced by Wassily Hoeffding in 1948,
 who established their theoretical foundation in their seminal paper \cite{hoeffding1948central}. A few classical examples of symmetric U-statistics are the unbiased estimators of the population variance and mean, the Wilcoxon--Mann--Whitney statistic for the median, the Shapiro--Wilk test for testing for normality, and Hoeffding's D statistic for testing the dependence of two random variables. 
 There are many important uses for U-statistics in applied statistics, including finding minimum-variance unbiased estimators for various classes of parameters. See Lee \cite{lee2019u} for a detailed introduction to U-statistics.
 }.

We make the following assumptions:

\begin{enumerate}
    \item (Mixing) There exists $c_{\text{mix}} < \infty$ such that $\phi(h)\leq c_{\text{mix}}h^{-p}$ for a large enough $p$.
    \item There exists $\rho>0$ such that $\sup_{N\in \mathbb{N}}\rho_3(N) \leq \rho<+\infty$.
    \item (Symmetry) $\psi(x,y) = \psi(y,x)$.
    \item (Degeneracy)  $\mathbb{E}\,[\psi_{i,j}(x, X_j)] = 0$ for all $x\in \mathcal{X}$.
    \item (Variance diverges) $\text{Var}\,(\sum_{j=1}^N g_j(X_j))\to \infty$.
    \item (Sub-linear variance growth) There exists $v_{\text{min}},v_{\text{max}}>0$ such that for all $N$ and every subset $\mathcal{I} \subseteq \{1,2, \cdots, N\}$, $$v_{\text{min}} \,\text{Card}\left(\mathcal{I}\right)^{\zeta} \leq \text{Var}\left(\sum_{j\in \mathcal{I}} g_j(X_j)\right)\le v_{\text{max}} \,\text{Card}\left(\mathcal{I}\right).$$
    \item $$\sup_{N\in \mathbb{Z}^{+}} \sup_{1\leq i < j \leq N}\left(\mathbb{E}\left[\psi_{i,j}^2(X_i, X_j)\right] + \mathbb{E}\left[\psi_{i,j}^2(\hat{X}_i, X_j)\right]\right)\leq L < +\infty,$$ where $\hat{X}_j$ is a copy of $X_j$ that is independent of $(X_n)$.
\end{enumerate}

We begin by using the same block decomposition as in Theorem \ref{Non-uniform third moment subsection}, forming blocks of variance $\sigma_N^{\alpha}$, $Y_j = \frac{1}{\sigma_N^{\alpha/2}}\sum_{i\in B_j} g(X_j)$, and gaps $G_j$ of variance $\sigma_N^{\beta}$, $0<\beta < \alpha < 2$, where the optimal choice of $\beta$ depends on $p$. This lets us define $$U=S+R, \text{ where } S = \sum_{j=1}^{k_N} Y_j \text{ and } R= \frac{1}{N}\sum_{1\leq i < j\leq N} \psi_{i,j}(X_i, X_j).$$

Note that since $\sigma_N^2 \ll N$, we have $\text{Var}(\sum_{j\in \mathcal{I}} g_j(X_j))\leq v_{\text{max}}\, \text{Card}(\mathcal{I})$. Furthermore, for each block $B_r$, $\text{Var}(\sum_{j\in B_r}g_j(X_j)) \geq \sigma_N^{\alpha}$ implies that $\text{Card}(B_r) \geq \sigma_N^{\alpha}/v_{\text{max}}$. Similarly, $\text{Card}(G_r)\geq \sigma_N^\beta/v_{\text{max}}$, which ensures approximate independence between the blocks via mixing.

\begin{theorem}[Berry-Esseen Bounds for U-Statistics]\label{U-statistics theorem}
    Assume $\psi_{i,j}$, $S$, and $R$ are as defined and satisfy the above assumptions. Then, there exists $C_{\alpha,\beta}>0$ independent of $N$ such that for all $\delta>0$ and $p$ large enough, $$\sup_{x\in \mathbb{R}}\left|P(U'<x\,\sigma_N) - \Phi(x)\right|\leq  C_{\alpha,\beta,\delta}\, \sigma_N^{2\beta/3 - 5\alpha/6 + 2/3 + \varepsilon_p} + C_{\alpha,\beta,\delta}\,\sigma_N^{-1-\delta p+2/\zeta}+C_{\alpha,\beta,\delta}\sigma_N^{-1+\delta}$$
\end{theorem}

\textbf{Proof}: We bound $\mathbb E[R]$ and $\gamma$, and then apply Theorem \ref{Berry-Esseen Bounds for Statistics of Weakly Dependent Samples Proof} to obtain rates of convergence.

\textbf{Bounding $\mathbb E[R]$}: We use a distance-based decomposition $R = R_{\le m} + R_{>m}$, where $m:=\sigma_N^{\delta}$ for a small $\delta>0$, $R_{\leq m} := \frac{1}{N}\sum_{i<j,\, |i-j|\leq m} \psi_{i,j}(X_i, X_j)$ and $R_{> m} := \frac{1}{N}\sum_{i<j,\, |i-j|> m} \psi_{i,j}(X_i, X_j)$.

\textbf{Bounding $\mathbb E[R_{\leq m}]$}: By Jensen's inequality and because there are $Nm$ summands, we have \begin{align}
    \mathbb E[|R_{\leq m}|] &\leq \frac{1}{N} \sum_{i<j, \, |i-j|\le m} \mathbb E [\psi^2_{i,j}]^{1/2}\\
    &\leq C \frac{1}{N } \cdot Nm = Cm = C\sigma_N^{\delta}.
\end{align}

\textbf{Bounding $\mathbb E[R_{> m}]$}: Since there are at most $N^2$ indices, we have 
\begin{align}
    \mathbb E \left[|R_{>m}|\right] &= \frac{1}{N}\sum_{i<j,\, |i-j|> m} \psi(X_i, X_j)\\
    &\leq CN \phi(m)\\
    &\leq C\sigma_N^{2/\zeta}\,\sigma_N^{-p\delta}\label{Distant R bounds}.
\end{align} Inequality \ref{Distant R bounds} follows from $N^{\zeta }\ll\sigma_N^2$ and the mixing rates. This is small for $p$ large enough.

Combining these bounds, for $p$ large enough, we have $\mathbb E[|R|] \leq C(\sigma_N^{2/\zeta}\,\sigma_N^{-p\delta}+\sigma_N^\delta)$.

\textbf{Bounding $\gamma$}: As shown in \cite{bentkus1997berry}, $$R_{j,k} = N^{-1}\sum_{\{ i,j \}\subseteq [\bigcup_{\ell=1}^{k_N} B_\ell]\setminus[B_1 \cup B_2 \cup \cdots \cup B_k]} \psi_{i,j}(X_i, X_j).$$  This gives,
$$R_{j,k} - R_{j,k-1} = \frac{1}{N}\sum_{i\in B_k}
\sum_{\ell\notin \bigcup_{r=j}^{k-1} B_s}
\psi_{i,j}(X_i,X_\ell).$$

 We also may assume that $|\psi(X_i,X_\ell)|\le \sigma_N^{6}$, as in \cite{bentkus1997berry}.

Recall that H\"older's inequality applied to a $L^2$ random variable $X$ gives $\mathbb E[|X|^{3/2}]\leq (\mathbb E[X^2])^{3/4}$. Furthermore, for each $j,k$ we have \begin{align}
    \mathbb E \left[|R_{j,k}-R_{j,k-1}|^{3/2}\right]
&=
N^{-3/2}\,\mathbb E \mathlarger\,[|\sum_{i\in B_k}
\sum_{\ell\notin \bigcup_{r=j}^{k-1} B_s}
\psi_{k,\ell}(X_k,X_\ell)|^{3/2}\mathlarger]\\
&\leq N^{-3/2}\left(\sum_{\ell_1,\ell_2\notin[j,k]}
\left|
\mathbb E\!\left[
\psi_{k,\ell_1}(X_k,X_{\ell_1})\,\psi_{k,\ell_2}(X_k,X_{\ell_2})
\right]
\right|\right)^{3/4}\\
&\leq CN^{-3/2} (N^2 \sigma_N^{6} \,\phi(\sigma_N^{\beta}))^{3/4}\label{conditioning on Xj inequality}\\
&\leq C\sigma_N^{9/2}\sigma_N^{-3p\beta/4}.%\\
%&\leq C \sigma_N^{-3/2}\sigma_N^{9/2}\sigma_N^{-3p\beta/4}.\label{Applying sublinear variance inequality}
\end{align}
Inequality \ref{conditioning on Xj inequality} follows from the degeneracy assumption and by conditioning $\psi_{k,\ell_1}(X_k,X_{\ell_1})\,\psi_{k,\ell_2}(X_k,X_{\ell_2})$ on $X_k$, and then applying Lemma \ref{Decoupling}. %Inequality \ref{Applying sublinear variance inequality} follows from $\sigma_N^2 \ll N$. 
Furthermore, this gives $$\gamma \leq C\left(\sigma_N^{9/2-3p\beta/4}\right)^{2/3},$$ which can be made arbitrarily small for $p$ large enough (or balanced with the other terms through the choice of $\beta$).

Finally, we consider the remainder arising from the gap terms. We bound these terms the same way as in Corollary \ref{Non uniform third moment corollary}, yielding the same rates.

Combining these bounds and applying Theorem \ref{Berry-Esseen Bounds for Statistics of Weakly Dependent Samples Proof} completes the proof.

\QEDB

\begin{remark}
    The phrase ``$p$ large enough" does a lot of heavy lifting in this theorem, as the value of $p$ affects the admissible choices of $\beta$. The smaller the choice of $\beta$, the larger the required value of $p$ (larger $\beta$ means weaker dependence between blocks). On the other hand, $\alpha$ can be chosen arbitrarily close to 2 with no cost of a larger required $p$. In fact, the theorem  could fail to provide non-trivial Berry--Esseen rates if $p$ is too small. If, however, the sequence has exponential or sub-exponential $\phi$-mixing rates, then we can choose $\beta$ arbitrarily small, and still obtain near-optimal rates. We also did not track the exact choice of $\beta$ in Theorems \ref{Non Uniform Third Moment Application}, \ref{U-statistics theorem}, and \ref{Berry-Esseen Bounds for Functions of Sample Means Theorem}. If the readers wish to see what rates are guaranteed given a known $p$, they would need to trace back the proofs to optimize the convergence rates. We would need to do the same for the choice of $\beta$ in Theorems \ref{Non Uniform Third Moment Application} and \ref{Berry-Esseen Bounds for Functions of Sample Means Theorem}. Note that the optimal choices of $\beta$ may differ for Theorems \ref{Non Uniform Third Moment Application}, \ref{U-statistics theorem}, and \ref{Berry-Esseen Bounds for Functions of Sample Means Theorem}. 

    We used a fairly crude bound when bounding $\gamma$ in Corollary \ref{Non uniform third moment corollary}. If we were to consider the closest blocks (i.e., for each $i$, considering only interactions between $B_i$ with $B_{i-1}$ and $B_{i+1}$) separate from the rest of the blocks, we might be able to decrease the required choice of $\beta$. 
\end{remark}

\begin{remark}
    It seems that we could also assume that $N^{\zeta_{\text{min}}}\ll\sigma_N^2 \ll N^{\zeta_{\text{max}}}$ for $\zeta_{\text{max}} >1$. The main way the upper bound is used is to ensure that the blocks are large enough, which this weaker condition would still ensure. However, this would require a larger choice of $p$.
\end{remark}

\begin{remark}
    If the sequence we are considering has exponential mixing rates, then we can modify Theorem \ref{U-statistics theorem} and apply \eqref{Main Theorem for exponential rates} (the generalization of \cite[Theorem 1.1]{bentkus1997berry}) to obtain faster Berry-Esseen rates. In this case, we could let the gap size $m = \sigma_N^{k}$ for $k$ large enough. Then, for gaps chosen of variance $\log (\sigma_N)$ and gaps of variance $\log (\sigma_N)^{1+k}$ for $k>0$, there exists an $r>0$ and $C>0$ such that for all $N\in \mathbb{Z}^+$,
 \begin{align*} \sup_{x\in \mathbb{R}}|P(T<\sigma_N x ) - \Phi(x)|\leq C\sigma_N^{-1}\log(\sigma_N)^{r}. \end{align*}
We could also obtain similar rates for Theorems \ref{Non Uniform Third Moment Application} and \ref{Berry-Esseen Bounds for Functions of Sample Means Theorem}.
 
\end{remark}

\subsection{Functions of Sample Means}\label{Functions of sample means section} This is a generalization of Theorem 2.1 of \cite{bentkus1997berry} to non-stationary statistics with non-linear variance. In this section, assume that $X_i$ take values in a real, separable Banach space $\mathcal{B}$, and define $\bar{X} := N^{-1}\sum_{j=1}^N X_j$. Let $f^{(j)}(x) h^j$ denote the $j$-th Fr\'{e}chet derivative of $f$ at point $x$ in direction $h$. Define $g(x) := H'(0)x$, and suppose suppose $\mathbb{E}[g(X_j)] = 0$ for all $j, N$ (note that $g$ depends on $N$ but not on $j$). Let $$R=N(H(\bar{X})-H(0)-H'(0) \bar{X}),$$ and $$T = \sum_{j=1}^N g(X_j) + R.$$ Also, let $H: \mathcal{B}\to \mathbb{R}$ be a Fr\'{e}chet differentiable function, and denote $M_s:= \sum_{j=1}^s \sup_{x\in B} \|H^{(j)}(x)\|$.

We make the following assumptions:
\begin{enumerate}
    \item (Mixing) $\phi(j)\leq Kj^{-p}$ for all $j\in \mathbb{Z}^+$.
    \item (Smoothness) Either $M_3 < \infty$ \textit{or} $\mathcal{B}$ is both a type 2 Banach space and $M_2 < \infty$.
    \item There exists $\rho>0$ such that $\sup_{N\in \mathbb{N}}\rho_3(N) \leq \rho<+\infty$.
    \item (Sub-linear variance growth) There exists $v_{\text{min}},v_{\text{max}}>0$ such that for all $N$ and every subset $\mathcal{I} \subseteq \{1,2, \cdots, N\}$, $$v_{\text{min}}\, \text{Card}\left(\mathcal{I}\right)^{\zeta} \leq \text{Var}\left(\sum_{j\in \mathcal{I}} g(X_j)\right)\le v_{\text{max}}\, \text{Card}\left(\mathcal{I}\right).$$
\end{enumerate}

\begin{theorem}[Berry-Esseen Bounds for Functions of Sample Means]\label{Berry-Esseen Bounds for Functions of Sample Means Theorem}
    Make the above assumptions. Then, there exists $C_{\alpha, \beta}>0$ independent of $N$ such that for $p$ large enough, \begin{align}
        \sup_x |\bbP(T<x) - \Phi(x)| \leq  C_{\alpha,\beta,\delta}\, \sigma_N^{2\beta/3 - 5\alpha/6 + 2/3 + \varepsilon_p} + C_{\alpha,\beta,\delta}\,\sigma_N^{-1-\delta p+2/\zeta}+C_{\alpha,\beta,\delta}\sigma_N^{-1+\delta}.\label{Functions of Sample Means rates}
    \end{align}
\end{theorem}

\textbf{Proof}: The proof involves a variance decomposition using the same variance threshold construction as in Theorem \ref{U-statistics theorem}. The rest of the proof follows the arguments of Theorem 2.1 of \cite{bentkus1997berry} with the blocked summands. 
\QEDB

\subsection{Studentized Sample Means}  Let $\{X_i\}_{i\in \mathbb{Z}^+}$ be a sequence of $\phi$-mixing, real-valued random variables with polynomially fast mixing rates (for a sufficiently large $p$), and such that  $\sigma_N^2\asymp N$. Also assume that $\mathbb{E}[X_1]=0$ and $\mathbb{E}[X_1^4]\leq \rho_4<+\infty$. Define $S_N = \sum_{j=1}^N X_j$ and $\sigma_N^2:=\mathbb{E}[S^2]$. Assume that $\lim_{N\to \infty}\sigma_N^2 >0$. Consider the estimator, $$s^2:= \sum_{j=1}^N \sum_{\ell: |j-\ell|\leq m}X_jX_{\ell},$$ of $\sigma$. Now define the standard deviation estimator, $$s := \begin{cases}
    \sqrt{s^2} &s\geq 0\\
    0 & s<0
\end{cases},$$ and define the studentized statistic, $$\Tilde{t} := \begin{cases}
    S/s &s>0\\
    0 & s=0
\end{cases}.$$ $\Tilde{t}$ is consistent and asymptotically unbiased for $m = C_0 N^{\epsilon_p}$ for a sufficiently large constant $C_0$ and $\epsilon_p\to 0$ \footnote{Studentization refers to a standardization of a statistic by dividing a sample statistic by a sample-based estimate of the population standard deviation. Studentized range, studentized residuals, and the Student $t$-statistic are classical examples, and are used when the parameter is unknown. ``Student" refers to the pseudonym of William Gosset, the statistician who developed the Student $t$-distribution.}. The following theorem follows immediately from the corresponding theorem in Bentkus et al. \cite{bentkus1997berry}. No efforts are made to extend this theorem to non-linear variance settings.
 \begin{theorem}[Berry-Esseen Bounds for Studentized Sample Means]
    For $p$ large enough, there exists a constant $A = A(K,  \sigma, \rho_4,p)$ independent of $N$ and $\varepsilon_p\to 0$ such that $$\sup_x |P(\tilde{t}<x) - \Phi(x)| \leq AN^{-1/2 + \varepsilon_p}.$$
\end{theorem} 

Optimal rates (in the stationary case) have already been attained if we were to assume that the variables are independent. See, for example, the proof by Jing et al. \cite{jing2000berry} of optimal Berry-Esseen bounds for a class of studentized statistics, including the studentized sample mean. 

The proof of this theorem and the other applications in Bentkus et al \cite{bentkus1997berry} when assuming linear variance growth for non-stationary processes are exempt, as their proofs are identical.

\section{Conclusion and Future Avenues} In this paper, we proved Berry-Esseen bounds for a broad class of statistics of sequences of weakly dependent random variables and then considered its applications. In Theorem \ref{Berry-Esseen Bounds for Statistics of Weakly Dependent Samples Proof}, various non-linear statistics are decoupled using the decoupling lemma, which requires $\phi$-mixing. An avenue for future research would be to prove Theorem \ref{Berry-Esseen Bounds for Statistics of Weakly Dependent Samples Proof} for $\alpha$-mixing sequences or another weaker form of mixing, or find counterexamples. Also, future research could determine in what applications it is possible to drop the linear variance assumption beyond the uniformly bounded case. This seems to require new methods already when $R=0$ since it is unclear when one can control the $L^3$ norms of blocks by means of the $L^2$ norms. 

Other applications of Theorem \ref{Berry-Esseen Bounds for Statistics of Weakly Dependent Samples Proof} to processes with non-linear variance growth can also be considered, such as V-statistics or linear combinations of order statistics (as in \cite{bentkus1997berry}), and small perturbations of stationary Markov chains, and also to processes in random environments (for which linear variance growth is expected). It is also an open question whether the $N^{\varepsilon_p}$ factors in the asymptotic bounds of Theorem \ref{Berry-Esseen Bounds for Statistics of Weakly Dependent Samples Proof} can be dropped (or even in the corresponding proof of Bentkus et al. \cite{bentkus1997berry}) even when $g_j$ are uniformly bounded; we did not attempt to find tight bounds for Theorem \ref{Berry-Esseen Bounds for Statistics of Weakly Dependent Samples Proof}. It would also be interesting to see how close to optimal the rates obtained for the applications are. On the other hand, tighter asymptotic bounds or a weaker form of mixing for the assumptions of stationarity and linear variance (as in Section \ref{Stationary extension section}) are another avenue for future research. Future research could also produce a simpler proof of Theorem \ref{Berry-Esseen Bounds for Statistics of Weakly Dependent Samples Proof} using the same or different methods of proof.

\section{Appendix}  This appendix contains proofs of results used in this paper.

We use the following lemma in the proof of Theorem \ref{Berry-Esseen Bounds for Statistics of Weakly Dependent Samples Proof}. The proof is for $\alpha$-mixing (strong mixing) sequences. In the proof of Theorem \ref{Berry-Esseen Bounds for Statistics of Weakly Dependent Samples Proof}, we apply the result to $\phi$-mixing sequences, which is valid because $\alpha(m)\leq \frac{1}{2}\phi(m)$ for all $m\in \mathbb{Z}^+$. Recall that the $m$-th $\alpha$-mixing coefficient is defined by $$\alpha(m):=\sup_{k\in \mathbb{Z}^{+}}\sup\left\{\left|\bbP(A\cap B) - \bbP(A)\bbP(B)\right|: A\in \sigma[k+m, +\infty), B\in\sigma[1,k] \right\}.$$ 

\begin{lemma}[Products in Strong Mixing Sequences]\label{Product Lemma} Let $m\in \mathbb{Z}^{+}$. Let $\{Y_j\}_{j=1}^{\infty}$ be an $\alpha$-mixing sequence of random variables such that $\|Y_j\|_{L^\infty}\leq C_j$, and $Y_j\in \sigma[a_j, b_j]$ and $a_j< b_j < a_{j+1} - m$ for each $j$. Then, there exists a constant $A>0$ such that for each $d \in \mathbb{Z^{+}}$,
    $$\left|\mathbb{E}\left[\prod_{j=1}^d Y_j\right] - \prod_{j=1}^d \mathbb{E}[Y_j]\right|\leq A\left(\prod_{j=1}^d C_j\right)(d-1)\,\alpha(m).$$
\end{lemma}
\textbf{Proof}: This result may be proved by induction over $d$. We only show $d=1, 2, 3$. After proving these cases, the inductive process should be clear. The $d = 1$ case is trivial. For the case $d = 2$, by H\"{o}lder's inequality we have, \begin{align}
    \mathlarger{|}\mathbb{E}\left[Y_1 Y_2\right] - \mathbb{E}\left[Y_1\right]\,\mathbb{E}\left[Y_2\right]\mathlarger{|} &= \mathlarger{|}\mathbb{E}\left[Y_1\,(Y_2 - \mathbb{E}[Y_2])\right]\mathlarger{|}\\
    &=\mathlarger{|}\mathbb{E}\left[Y_1\,\mathbb{E}\left[Y_2 - \mathbb{E}\left[Y_2\right]\,|\,\sigma(Y_1)\right]\right]\mathlarger{|}\\
    &\leq \left\|Y_1\right\|_{L^\infty}\left\|\mathbb{E}\left[Y_2 - \mathbb{E}\left[Y_2\right]\,|\,\sigma(Y_1)\right]\right\|_{L^1}\\
    &\leq AC_1C_2\,\alpha(m).
\end{align} For $d = 3$, by the triangle inequality we may write, \begin{align}
    \mathlarger{|}\mathbb{E}[Y_1 Y_2 Y_3] - \mathbb{E}[Y_1]\mathbb{E}[Y_2] \mathbb{E}[Y_3]\mathlarger{|} &=\mathlarger{|}\mathbb{E}[(Y_1 Y_2) Y_3] - \mathbb{E}[Y_1]\mathbb{E}[Y_2] \mathbb{E}[Y_3] - \mathbb{E}[Y_1Y_2] \mathbb{E}[Y_3]+\mathbb{E}[Y_1Y_2] \mathbb{E}[Y_3]\mathlarger{|}\\
    &\leq \mathlarger{|}\mathbb{E}[Y_1 Y_2 Y_3] -  \mathbb{E}[Y_1Y_2] \mathbb{E}[Y_3]\mathlarger{|} + \mathlarger{|}\mathbb{E}[Y_1Y_2] \mathbb{E}[Y_3] - \mathbb{E}[Y_1]\mathbb{E}[Y_2] \mathbb{E}[Y_3]\mathlarger{|}.
\end{align} Furthermore, 
\begin{align}
    \mathlarger{|}\mathbb{E}[Y_1Y_2] \,\mathbb{E}[Y_3] - \mathbb{E}[Y_1]\,\mathbb{E}[Y_2] \,\mathbb{E}[Y_3]\mathlarger{|}\leq AC_1C_2C_3\,\alpha(m)
\end{align} by the $d=2$ case. This thus gives, \begin{align}
    \mathlarger{|}\mathbb{E}\left[Y_1 Y_2 Y_3\right] -  \mathbb{E}[Y_1Y_2]\, \mathbb{E}[Y_3]\mathlarger{|}  &= \mathlarger{|}\mathbb{E}\left[Y_1 Y_2\, \mathbb{E}[Y_3 - \mathbb{E}[Y_3]\,|\,\sigma(Y_1, Y_2) ]\right]\mathlarger{|}\\
    &\leq AC_1C_2C_3\,\alpha(m).\label{d three induction}
\end{align} Inequality \ref{d three induction} follows from the $d=2$ case.  Summing these results concludes the $d=3$ case. The inductive step can be established using the same method as in the $d=3$ case. \QEDB

The following is a Theorem from \cite{rio1993covariance} that provides explicit covariance bounds for $\alpha$-mixing sequences.

\begin{theorem}[Rio, 1993]\label{Rio Covariance Theorem} Suppose $\{X_i\}_{i\in\mathbb{Z}^+}$ is an $\alpha$-mixing sequence, and for random variables $X$, define $\alpha(i,j) := \alpha(|i-j|)$ and the quantile function $Q_X(u):=\inf \{t>0: \bbP(|X| > t) < u\}$. Then, for each $i,j\in \mathbb{Z}^+$,
    $$\left|\text{Cov}\left(X_i,X_j\right)\right| \leq 2 \int_0^{\alpha(i,j)}\,Q_{X_i} (u)\, Q_{X_j} (u)\,du.$$
\end{theorem}

The following is a useful corollary of Theorem \ref{Rio Covariance Theorem} for us.

\begin{corollary}[Mixing Covariance Bound]\label{Rio Covariance Corollary} Let $\{X_i\}$ be a centered, $\alpha$-mixing (or $\phi$-mixing) sequence and suppose that $\sup_{i\in\mathbb{Z}^+} \mathbb{E}[|X_i|^3]\leq \rho$. Then for each $i,j$,  $$|\text{Cov}\,(X_i, X_j)| \leq 6\rho^{2/3} \alpha(i,j)^{1/3}.$$
\end{corollary}
\textbf{Proof}: By Markov's inequality, we have that \begin{align}
    Q_{X_i}(u) &= \inf \{t>0: \bbP(|X_i| > t) < u\}\\
    &\leq \inf \{t>0: \frac{\|X_i\|_{L^3}^3}{t^3}\leq u\}\\
    &\leq \frac{\|X_i\|_{L^3}}{u^{1/3}}\leq \frac{\rho^{1/3}}{u^{1/3}}.
\end{align}

By Theorem \ref{Rio Covariance Theorem}, \begin{align}
    |\text{Cov}(X_i, X_j)|&\leq 2\int_0^{\alpha(i,j)} \left(\frac{\rho^{1/3}}{u^{1/3}}\right)\left(\frac{\rho^{1/3}}{u^{1/3}}\right)du\\
    &= 6\rho^{2/3}\, (\alpha(i,j))^{1/3}.
\end{align}

 \QEDB

\section{Theorem \ref{Berry-Esseen Bounds for Statistics of Weakly Dependent Samples Proof} Truncation Details}\label{TruncSec}

In Theorem \ref{Berry-Esseen Bounds for Statistics of Weakly Dependent Samples Proof} (and similarly in some of its applications), we replace $\mathbb{E}\,[g_j(X_j)]$ with the truncated and centered random variable $\mathbb{E}\,[g_j(X_j)\mathbbm{1}_{\{|g_j(X_j)|\leq \sqrt{N}\}}]$. In this section, we add details for why this substitution is justified. 

When we replace $\mathbb{E}\,[g_j(X_j)]$ with $\mathbb{E}\,[g_j(X_j)\mathbbm{1}_{\{|g_j(X_j)|\leq \sqrt{N}\}}]$ in Theorem \ref{Berry-Esseen Bounds for Statistics of Weakly Dependent Samples Proof}, we introduce a few sources of error. We bound each source of error individually to how that truncation by $\sqrt{N}$ introduces only negligible $O(N^{-1/2})$ error, so the resultant Berry-Esseen bounds are unchanged.

%All three sources of error arise in the Taylor expansion. The single-factor Taylor expansion is $e^{it g(X_j)^{(B)}/\sqrt{N}} = 1 + \frac{it}{\sqrt{N}}g(X_j)^{(B)} - \frac{t^2}{2N}(g(X_j)^{(B)})^2 + R_j^{(B)}(t)$, where $R_j^{(B)}(t):= \sum_{k=3}^{\infty} \frac{(itg(X_j)^{(B)}/ \sqrt{N})^k}{k!}$ is the remainder.

\textbf{Shift Error}: The centering shift may change after truncation, but we show that it can be absorbed into $R$ without affecting the Berry-Esseen rates. Define 
$$\widetilde g_j(X_j)
:= g_j(X_j)\mathbbm 1_{\{|g_j(X_j)|\le \sqrt{N}\}}
- \mathbb E\big[g_j(X_j)\mathbbm 1_{\{|g_j(X_j)|\le \sqrt{N}\}}\big], \,\,\, \widetilde S_N := \sum_{j=1}^N \widetilde g_j(X_j),\,\,\,
\widetilde T_N := \widetilde S_N + \widetilde R_N.$$
 If we let $\Delta_j := g_j(X_j)-\widetilde g_j(X_j) = g_j(X_j)\mathbbm 1_{\{|g_j(X_j)|>\sqrt{N}\}}
- \mathbb E[g_j(X_j)\mathbbm 1_{\{|g_j(X_j)|>\sqrt{N}\}}]$, then the shift error is $S_N-\widetilde S_N=\sum_{j=1}^N \Delta_j$. By the uniform third moment bound, for each $j$ we have, \begin{align*}
    \mathbb{E}\left[|\Delta_j|\right]&\leq 2\,\mathbb E\left[|g_j(X_j)|\,\mathbbm 1_{\{|g_j(X_j)|>\sqrt{N}\}}\right]\\
    &\leq 2\,\mathbb E\left[|g_j(X_j)|^3\right] \, N^{-1}
\;\leq\; \rho N^{-1}.
\end{align*} Summing over all $j\in\{1,\cdots, N\}$, we see that the shift error is negligible.

%\textbf{Event Difference} Let $d_K$ denote the Kolmogorov distance. By Markov's inequality, $d_K(S, S_N^{(B)}) = \sup_{x\in \mathbb{R}}|P(S_N\leq x) - P(S_N^{(B)}\leq x)| \leq P(S \neq S_N^{(B)}) \leq P(\{\exists j: |g(X_j) > B|\}) \leq \sum_{j=1}^N P(|g(X_j)|>B)\leq \sum_{j=1}^N \frac{\mathbb{E}[|g(X_j)|^3]}{B^3}$. Note that this is $O(N^{-1/2})$ for $B = \sqrt{N}$. 

\textbf{Variance Change}: Truncation may also change the variance, but we similarly show that this change is negligible. We may write
\begin{align*}
\text{Var}(S_N)
&=\text{Var}\!\left(S_N^{(\sqrt N)}\right)
+ \text{Var}\left(\sum_{j=1}^N\Delta_j\right)
+ 2\,\text{Cov}\!\left(S_N^{(\sqrt N)},\Delta_N\right),
\end{align*} so $$\mathlarger{|}\text{Var}(S_N)-\text{Var}(S_N^{(\sqrt N)})\mathlarger{|}
\;\le\;
\text{Var}\left(\sum_{j=1}^N\Delta_j\right)
+2\,\mathlarger{|}\text{Cov}\mathlarger{(}S_N^{(\sqrt N)},\sum_{j=1}^N\Delta_j\mathlarger{)}\mathlarger{|}.$$

Since $|\Delta_j|
\le
2\,|g_j(X_j)|\mathbbm 1_{\{|g_j(X_j)|>\sqrt N\}}$ for each $j$, we have 
$$\mathbb E\left[\Delta_j^2\right]
\;\le\;
4\,\mathbb E\,\mathlarger{[}|g_j(X_j)|^2\,\mathbbm 1_{\{|g_j(X_j)|>\sqrt N\}}\mathlarger{]}
\;\le\;
4\,\frac{\mathbb E\left[|g_j(X_j)|^3\right]}{\sqrt N}
\;\le\;
\frac{4\rho}{\sqrt N}.$$ Summing over $j$ and applying the triangle inequality completes the bound for $\text{Var}(\sum_{j=1}^N\Delta_j)$ (we must control the off-diagonal covariance terms via mixing, valid for $p$ large enough). The H\"older inequality completes the bound for $|\text{Cov}(S_N^{(\sqrt N)},\sum_{j=1}^N\Delta_j)|$.

This establishes that the variance change arising from truncation is also negligible.

Furthermore, it is also straightforward to check that the various statistical properties, such as the variance, $\mathbb{E}[R]$, $\gamma$, $R_{j,k}$ are only changed a negligible amount from truncation.

\newpage
\printbibliography

\end{document}